\documentclass[12pt]{amsart} %
\usepackage[top=1in, bottom=1in, left=1in, right=1in]{geometry} 
\usepackage{float}
\usepackage{placeins}

\newcommand{\ove}{\overline}

\newcommand{\what}{\widehat}
\newcommand{\sbs}{\subseteq}
\newcommand{\sps}{\supseteq}
\newcommand{\pa}[1]{\left(#1\right)}
\usepackage{tikz}
\usepackage[backref=page]{hyperref}
\usepackage{cleveref}
\usepackage{breqn}
\usepackage{titletoc}

\DeclareFontFamily{U}{shuffle}{}
\DeclareFontShape{U}{shuffle}{m}{n}{%
<5-8>shuffle7%
<8->shuffle10
}{}
\DeclareSymbolFont{Shuffle}{U}{shuffle}{m}{n}
\DeclareFontSubstitution{U}{shuffle}{m}{n}
\DeclareMathSymbol\shuffle{\mathbin}{Shuffle}{"001} \DeclareMathSymbol\cshuffle{\mathbin}{Shuffle}{"002}

\usepackage{graphicx}
\usepackage{color} 
\usepackage{amsmath, amsfonts, amssymb, amsxtra, amsthm}
\usepackage{mathrsfs}
\usepackage{dsfont}
\usepackage{verbatim}
\usepackage{comment}
\usepackage{enumitem}
\newlist{casesp}{enumerate}{3} 
\setlist[casesp]{align=left, 
                 listparindent=\parindent, 
                 parsep=\parskip, 
                 font=\normalfont\bfseries, 
                 leftmargin=0pt, 
                 labelwidth=0pt, 
                 itemindent=.4em,labelsep=.4em, 
                 partopsep=0pt, 
                 }
\setlist[casesp,1]{label=Case~\arabic*:,ref=\arabic*}
\setlist[casesp,2]{label=Case~\thecasespi.\arabic*:,ref=\thecasespi.\arabic*}
\setlist[casesp,3]{label=Case~\thecasespii.\arabic*:,ref=\thecasespii.\arabic*}
\DeclareMathOperator{\del}{del}
\DeclareMathOperator{\sh}{sh}

\usepackage{indentfirst}
\numberwithin{equation}{section}

\usepackage{amssymb} 
\usepackage{hyperref} 
\hypersetup{colorlinks = true,
    linkcolor={red},
    citecolor={blue},
    urlcolor={magenta}} 
\usepackage{tikz} 

\newcommand{\Z}{\mathbb Z} 
\newcommand{\sm}{\setminus} 
\newcommand{\Fcal}{\mathcal F} 
 
\newcommand{\Hc}{\mathcal H} 
\newcommand{\R}{\mathbb R} 
\newcommand{\bfx}{\mathbf x}

\DeclareMathOperator{\Des}{Des}
\DeclareMathOperator{\SSYT}{SSYT}
\DeclareMathOperator{\Sym}{Sym}
\DeclareMathOperator{\QSym}{QSym}
\DeclareMathOperator{\SYT}{SYT}
\DeclareMathOperator{\std}{std}

\theoremstyle{plain} 
\newtheorem{thm}{Theorem}[section]
\newtheorem{prop}[thm]{Proposition} 
 
\newtheorem{lm}[thm]{Lemma}
\newtheorem{cor}[thm]{Corollary}

\newtheorem{ex}[thm]{Example}

\theoremstyle{definition} 
\newtheorem{defn}[thm]{Definition} 
\newtheorem{rmk}[thm]{Remark} 

\usepackage{mathtools}
\DeclarePairedDelimiter\abs{\lvert}{\rvert}
\DeclarePairedDelimiter\floor{\lfloor}{\rfloor}
\DeclarePairedDelimiter\ceil{\lceil}{\rceil}

\title{On symmetric pattern avoidance sets}
\author{Tuong Le}

\address[Le]{
Dept.\ of Mathematics\\
   Princeton University\\
   Princeton, NJ 08544, USA
}
\email{tl0101@princeton.edu}

\date{\today}

\begin{document}

\begin{abstract}
    For a set of permutations $S\sbs S_n$, consider the quasisymmetric generating function 
    \[Q(S): = \sum_{w\in S}F_{n, \Des(w)},\]
    where $\Des(w) := \{i\mid w(i)> w(i+1)\}$ is the descent set of $w$ and $F_{n, \Des(w)}$ is Gessel's fundamental quasisymmetric function. A set of permutations is said to be symmetric (respectively, Schur-positive) if its quasisymmetric generating function is symmetric (respectively, Schur-positive). Given a set $\Pi$ of permutations, let $S_n(\Pi)$ denote the set of permutations in $S_n$ that avoid all patterns in $\Pi.$ A set $\Pi$ is said to be symmetrically avoided (respectively, Schur-positively avoided) if $S_n(\Pi)$ is symmetric (respectively, Schur-positive) for all $n.$

    Marmor proved in 2025 that for $n\ge 5$, a symmetric set $S\sbs S_n$ has size at least $n-1$ unless $S\sbs \{12\cdots n, n\cdots 21\}$ and asked for a general classification of the possible sizes of symmetric sets not containing the monotone elements $12\cdots n $ and $n\cdots 21$. We give a complete answer to this question for $n\ge 52.$ 
    We also give a classification of symmetric sets of size at most $n-1$, thereby showing that they are actually Schur-positive, resolving a conjecture of Marmor. 
    Finally, we give a classification of symmetrically avoided sets of size at most $n-1$, thereby showing that they are actually Schur-positively avoided.
\end{abstract}

\maketitle

\section{Introduction}
For a set of permutations $S\subseteq S_n$, consider the quasisymmetric generating function 
\begin{equation}
    Q(S):= \sum_{w\in S}F_{n, \Des(w)},
\end{equation}
where $\Des(w) := \{i\mid w(i) > w(i+1)\}$ is the descent set of $w$ and $F_{n, \Des(w)}$ is Gessel's fundamental quasisymmetric function (see Section \ref{ssec:symandqsym} for more definitions and background). Gessel's fundamental quasisymmetric functions were first introduced by Gessel in 1984 \cite{gessel}, and the problem of finding sets of permutations for which the associated quasisymmetric function is symmetric was first posed by Gessel and Reutenauer in 1993 \cite{gesselreutenauer}. We say a set $S$ is \textit{symmetric} if $Q(S)$ is symmetric, and $S$ is \textit{Schur-positive} if $Q(S)$ is Schur-positive. 

It is natural to ask for a classification of the possible sizes of symmetric sets. It turns out (Remark \ref{rmk:delmonotone}) that $S$ is symmetric if and only if $S\sm\{\iota_n, \delta_n\}$ is, where $\iota_n := 12\dots n$ and $\delta_n := n\dots 21$ are the monotone elements. Thus, it suffices to ask for the possible sizes of a symmetric set without the monotone elements (see, e.g., \cite[Question 5.3]{M25}). Marmor \cite{M25} showed in 2025 that for $n\ge 5$, if $S\sbs S_n$ is symmetric and $S\not\sbs \{\iota_n, \delta_n\}$ then $|S|\ge n-1.$ In this paper, we classify all the possible sizes of a symmetric set without monotone elements in $S_n$ for $n\ge 52.$ It turns out that $S$ is symmetric if and only if $S_n\sm(S\cup \{\iota_n, \delta_n\})$ is (Remark \ref{rmk:flipsize}), so it suffices to classify all the possible sizes up to $\frac{n!-2}2.$ The classification follows from the following two theorems.
\begin{thm}\label{thm:symsetlargsize}
    For $n\ge 4$ and $(n-1)(n-3)\le p\le \frac{n!-2}2,$ there is a symmetric set of size $p$ without monotone elements.
\end{thm}

\begin{thm}\label{thm:symsetsmallsize}
    For $n\ge 52$ and $p < (n-1)(n-3)$, there exists a symmetric set of size $p$ without monotone elements if and only if \textup{(}exactly\textup{)} one of the following is true
    \begin{enumerate}[label=\textup(\arabic*\textup)]
        \item $p = qn-r$ for some $0\le r \le q < n-3,$
        \item $n$ is even and $p = \frac{n(n-3)}2+c(n-1)$ for some $c\ge 0,$
        \item $n$ is odd and $p = \binom{n-1}2+cn$ for some $c\ge 0.$
    \end{enumerate}
\end{thm}
Marmor showed that their lower bound $n-1$ is tight by giving an example of a Schur-positive set of size exactly $n-1$, and conjectured that all symmetric sets realizing this lower bound are actually Schur-positive \cite[Conjecture 5.2]{M25}. We give a classification of symmetric sets of sizes at most $n-1$, which gives a new proof of Marmor's result and show that all symmetric sets of size at most $n-1$ are Schur-positive. 
\begin{thm}\label{thm:symmk}
    Suppose $S\subseteq S_n$ and $|S|\le n-1$. Then the following are equivalent:
    \begin{enumerate}[label=\textup(\arabic*\textup)]
        \item $S$ is symmetric.
        \item One of the following is true:
    \begin{enumerate}[label=\textup(\alph*\textup)]
        \item $S\subseteq \{\iota_n, \delta_n\}.$
        \item $S = \{\pi_1, \dots, \pi_{n-1}\}$ where either $\Des(\pi_i) = \{i\}$ for all $i$, or $\Des(\pi_i) = [n-1]\sm\{i\}$ for all $i$.
        \item $n = 4$ and $S = \{\pi_1, \pi_2\}$ or $S = \{\pi_1, \pi_2, \iota_4\}$ or $S = \{\pi_1, \pi_2, \delta_4\}$, where $\Des(\pi_1) = \{1, 3\}$  and $\Des(\pi_2) = \{2\}$.
        \item $n = 6$ and $S = \{\pi_1, \pi_2, \dots, \pi_5\}$ where $\Des(\pi_1) = \{1, 3, 5\},$ $\Des(\pi_2) = \{2, 5\},$ $\Des(\pi_3) = \{3\}, \Des(\pi_4) = \{1,4\}, \Des(\pi_5) = \{2, 4\}$; or $\Des(\pi_1) = \{2,4\},$ $\Des(\pi_2)$ $ = \{1,3,4\}$, $\Des(\pi_3) = \{1,2,4,5\},$ $\Des(\pi_4) = \{2,3,5\},$ $\Des(\pi_5) = \{1,3,5\}$.
    \end{enumerate}
    \item $S$ is Schur-positive.
\end{enumerate}
\end{thm}

A remarkable example of a symmetric set of permutations is the arc permutations \cite{er}, which is the set of all permutations in $S_n$ avoiding the pattern set $\Pi := \{w\in S_4\mid |w(1)-w(2)| = 2\}.$ This has motivated a line of work in finding sets of patterns $\Pi$ such that $S_n(\Pi)$, the set of all permutations in $S_n$ avoiding $\Pi$, is symmetric for all $n$ \cite{bloomsagan, ER17, HZP, M25, S15}. Such sets are called \textit{symmetrically avoided}. If $S_n(\Pi)$ is actually Schur-positive for all $n$, we say $\Pi$ is \textit{Schur-positively avoided}.

Bloom and Sagan \cite{bloomsagan} initiated the study of symmetrically avoided sets of small size. They proved that for $k\ge 4,$ if $\Pi\sbs S_k$ is symmetrically avoided and $|\Pi|\le 2,$ then $\Pi\sbs\{ \iota_k, \delta_k\}.$
Marmor \cite{M25} extended this result, showing that for $k\ge 5$, if $\Pi\sbs S_k$ is symmetrically avoided and $|\Pi|\not \sbs \{ \iota_k, \delta_k\},$ then $|\Pi|\ge k-1.$ 
We give a complete classification of all symmetrically avoided sets of size at most $k-1,$ which as a corollary shows that these sets are actually Schur-positively avoided.
\begin{thm}\label{thm:smallsaset}
    Given $\Pi\subseteq S_k$ such that $|\Pi|\le k-1$, the following are equivalent.
    \begin{enumerate}[label=\textup(\arabic*\textup)]
        \item $\Pi$ is symmetrically avoided,
        \item $\Pi$ is a partial shuffle or the complement of a partial shuffle, or $\Pi\subseteq \{\iota_k, \delta_k\},$
        \item $\Pi$ is Schur-positively avoided.
    \end{enumerate}
\end{thm}
Here, the word ``complement'' refers to the element-wise complement of a set, which we define in Section \ref{sec:background} (Definition \ref{def:permcr}). Partial shuffles were defined by Bloom and Sagan \cite{bloomsagan}; we recap this in Definition \ref{def:partialsh}.

Our proof technique involves recasting the problem of symmetric sets in terms of \textit{harmonic set systems}, which we define in Section \ref{sec:har} (Definition \ref{def:harmonicset}). Informally, a harmonic set system is a collection of sets such that the sizes of their intersections are uniform in a certain sense. Inspired by \cite{M25}, for each set of permutations one can define an associated set system, which we prove is harmonic if and only if the set of permutation is symmetric.

The rest of the paper is structured as follows. In Section \ref{sec:background} we go over relevant background and preliminaries. In Section \ref{sec:har}, we introduce harmonic set systems. In Section \ref{sec:sizeofsymm}, we prove Theorem \ref{thm:symsetlargsize} and Theorem \ref{thm:symsetsmallsize}. In Section \ref{sec:minsymmset}, we prove Theorem \ref{thm:symmk} and Theorem \ref{thm:smallsaset}. 
\subsection*{Acknowledgements}
This research was conducted at the University of Minnesota Duluth REU with support from Jane Street Capital, NSF Grant 1949884, and donations from Ray Sidney and Eric Wepsic. The author thanks Mitchell Lee for extensive feedback and guidance during the research and writing process, and Katherine Tung for helpful writing advice. The author also thanks Joe Gallian and Colin Defant for their support and for the invitation to participate in the Duluth REU, and Maya Sankar, Noah Kravitz
as well as other students and visitors for helpful discussions. 
\section{Background and preliminaries}\label{sec:background}

\subsection{Permutations and pattern avoidance}

We use the notation $[n]:=\{1, 2, \dots, n\}.$ Let $S_n$ be the symmetric group on $[n]$, i.e. the set of bijections from $[n]$ to $[n]$. One can think of a permutation $\sigma\in S_n$ as a sequence $\sigma(1), \dots, \sigma(n)$. Given any sequence $\sigma$ of $k$ distinct numbers, its standardization, denoted $\std(\sigma)$, is the sequence obtained by replacing the $i$-th smallest number in the sequence by $i$ for $1\le i \le k$.
\begin{defn}
    A permutation $w\in S_n$ is said to \textit{contain} $\pi\in S_k$ if there is a subsequence $\sigma$ of $w$ such that $\std(\sigma) = \pi$. It is said to \textit{avoid} $\pi$ otherwise. When considering permutations containing or avoiding $\pi$, we will sometimes refer to $\pi$ as a \textit{pattern}.
\end{defn}
\begin{ex}
The permutation $4\underline{1}52\underline{6}\underline{73}\in S_7$ contains $1342\in S_4$ since $1673$ is a subsequence \textup(underlined in the original sequence\textup) and $\std(1673) = 1342$. However, it avoids $3241$, since no subsequence standardize to $3241$.
\end{ex}
\begin{defn}
    Given $\Pi\subseteq \bigcup\limits_{k = 1}^\infty S_k$, we say $w\in S_n$ avoids $\Pi$ if it avoids all patterns $\pi \in \Pi$.
\end{defn}
\begin{defn}
    Given a pattern set $\Pi\subseteq \bigcup\limits_{k = 1}^\infty S_k$, we define $S_n(\Pi)$ to be the set of all permutations in $S_n$ avoiding $\Pi$:
    \begin{equation}
        S_n(\Pi) := \{w\in S_n\mid w\text{ avoids }\Pi\}
    \end{equation}
\end{defn}
\begin{defn}
    Define $\iota_k := 1,2,\dots, k\in S_k$ and $\delta_k = k, k-1,\dots, 1\in S_k$. We call $\iota_k, \delta_k$ \textit{the monotone elements} of $S_k$.
\end{defn}
It will be helpful to have some geometric perspective of pattern avoidance. We will roughly follow the convention as in \cite{ggc}.
\begin{defn}
    A \textit{figure} is a subset of $\R^2$. A \textit{plot} is a finite figure that is \textit{independent}; that is, no two points lie on the same horizontal or vertical line. Given a plot $P$, one can associate a corresponding permutation $\sigma$ as follows. Consider two ordering of the points from left to right (i.e. increasing $x$-coordinate), and from bottom up (i.e. increasing $y$-coordinate). For each point $p\in P$, let $x(p)$ and $y(p)$ be the position of $p$ in the ordering. Then $\sigma$ is the permutation that maps $x(p)$ to $y(p)$ for each $p$. We say $\sigma$ is the permutation of $P$ and $P$ is a plot of $\sigma$. For each permutation $\sigma\in S_n$, define its \textit{standard plot} to be $\{(i, \sigma(i))\mid 1\le i\le n\}$.
\end{defn}
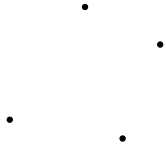
\begin{figure}[h!]
    \begin{tikzpicture}[scale=0.5]
        \filldraw (0,1) circle (2pt);
        \filldraw (2,4) circle (2pt);
        \filldraw (3,0.5) circle (2pt);
        \filldraw (4,3) circle (2pt);
    \end{tikzpicture}
    \caption{A plot of $2413$}\label{fig:plo2413}
\end{figure}
\begin{ex}
    Figure \ref{fig:plo2413} shows a plot of $2413$. 
\end{ex}
\begin{rmk}
    A permutation $w$ contains a pattern $\pi$ if and only if every plot of $w$ contains a plot of $\pi$. In that case, every plot of $\pi$ can be extended into a plot of $w$ by adding some points.
\end{rmk}

\subsection{Symmetric and quasisymmetric functions}\label{ssec:symandqsym}
Let $\mathbf x = \{x_1, x_2,\dots\}$, and let $\Z[[\mathbf x]]$ denotes the ring of formal power series in the variables $x_1, x_2, \dots$ with coefficients in $\Z$. An element $f\in \Z[[\bfx]]$ is said to be symmetric if it is invariant under swapping any two variables. Let $\Sym$ denote the ring of symmetric functions; i.e. symmetric elements $f\in \Z[[\bf x]]$ with bounded degree.

\begin{defn}
A partition $\lambda$ of $n$ is a weakly decreasing sequence of positive integers $\lambda = (\lambda_1\ge  \dots\ge  \lambda_k)$ that adds up to $n$, i.e. $\lambda_1+\dots +\lambda_k = n$. If $\lambda$ is a partition of $n$ we write $\lambda\vdash n$ or $\abs*{\lambda} = n$.
\end{defn}
A partition $\lambda$ has an associated Young diagram, which has $\lambda_i$ boxes on row $i$ for $1\le i \le k$.  We will draw Young diagrams using English notation, e.g. the first row is on top, and the boxes in each row are left-justified. We will usually identify a partition with its diagram. The conjugate partition  of $\lambda$, denoted $\lambda^t$, is the partition whose Young diagram is obtained by flipping the Young diagram of $\lambda$ along the main diagonal (and thus interchanging rows and columns). A semistandard Young tableau (SSYT) of shape $\lambda$ is a filling of the boxes of $\lambda$ such that numbers in the same row are weakly increasing and numbers in the same column are strictly increasing. If $T$ has shape $\lambda$ we write $\sh(T) = \lambda$. A standard Young tableau (SYT) of shape $\lambda$ is a SSYT of shape $\lambda$ such that each number from $1$ to $|\lambda|$ appears exactly once. Let $\SSYT(\lambda)$ and $\SYT(\lambda)$ denotes the set of SSYT and SYT of shape $\lambda$, respectively. 
\begin{defn}
For each $\lambda$, let $s_\lambda$ denote the corresponding Schur function, that is 
\begin{equation}
    s_\lambda:= \sum\limits_{T\in \SSYT(\lambda)}x^T.
\end{equation}
\end{defn}
The Schur functions form a $\Z$-basis for $\Sym$ as a $\Z$-module. A symmetric function $f$ is said to be \textit{Schur-positive} if its expansion in the Schur basis has only nonnegative coefficients.  

Quasisymmetric functions are a generalization of symmetric functions. We say $f\in \Z[[\bfx]]$ is quasisymmetric if for all $i_1 < \dots < i_k$, all $j_1 < \dots < j_k$, and all $a_1, a_2, \dots, a_k> 0$, the coefficient of $x_{i_1}^{a_1}\cdots x_{i_k}^{a_k}$ in $f$ is the same as the coefficient of $x_{j_1}^{a_1}\cdots x_{j_k}^{a_k}$. Let $\QSym$ denotes the ring of quasisymmetric functions; i.e. quasisymmetric elements $f\in \Z[[\bfx]]$ with bounded degree.

\begin{defn}
A composition $\alpha$ of $n$ is a sequence of positive integers $\alpha= (\alpha_1, \dots, \alpha_k)$ such that $\alpha_1 + \dots + \alpha_k = n$. If $\alpha$ is a composition of $n$ we write $\alpha \vDash n$ or $\abs{\alpha}  = n$.
\end{defn}

There is an explicit bijection between compositions of $n$ and subsets of $[n-1]$, given by 
\begin{equation}
(\alpha_1, \dots, \alpha_k)\mapsto \{\alpha_1, \alpha_1 + \alpha_2 , \dots, \alpha_1 + \alpha_2 + \dots+ \alpha_{k-1}\}
\end{equation}
The subset of $[n-1]$ corresponding to $\alpha$ will be denoted $S_\alpha$. Given $S\subseteq [n-1]$, the corresponding composition of $n$ will be denoted $\alpha_S$.

\begin{defn}\label{def:fundamentalquasi}
For each subset $S\subseteq [n-1]$, let $F_{n,S}$ denote the corresponding fundamental quasisymmetric function:
\begin{equation}
    F_{n, S} := \sum_{\substack{i_1\le \dots \le i_n\\ i_j < i_{j+1}\text{ if }j\in S}}x_{i_1}\cdots x_{i_n}.
\end{equation}
\end{defn}
We may omit the $n$ and write $F_S$ if the $n$ is clear from context. The fundamental quasisymmetric functions form a $\Z$-basis for $\QSym$ as a $\Z$-module. Since the Schur functions are symmetric, they are quasisymmetric, and thus can be expanded in this basis. To explain exactly how they expand, we need the following definition.
\begin{defn}
    Given a SYT $Q$, define its descent set to be $\Des(Q):= \{i\mid i+1$ is on a lower row than $i\text{ in }Q\}$.
\end{defn}
\begin{thm}[{\cite[Theorem~7.19.7]{stanley}}]\label{thm:slambda} We have 
    \begin{equation}
        s_\lambda = \sum_{Q\in\SYT(\lambda)}F_{\Des(Q)}.
    \end{equation}
\end{thm}

\subsection{Symmetric and Schur-positive sets}
\begin{defn}
    Given an injective function $w\colon [n]\to \R$ (for example, a permutation in $S_n$), we define its descent set 
    \begin{equation}
    \Des(w):= \{i\mid w(i) > w(i+1)\}.
\end{equation}
\end{defn}
\begin{defn}
    Given a set $S\subseteq S_n$, we define its  quasisymmetric generating function 
\begin{equation}
    Q(S):= \sum_{w\in S}F_{\Des(w)}.
\end{equation}
\end{defn}
\begin{defn}
    A set $S\subseteq S_n$ is said to be \textit{symmetric} if $Q(S)$ is a symmetric function. It is said to be \textit{Schur-positive} if $Q(S)$ is a symmetric function that is Schur-positive.
\end{defn}
\begin{rmk}
    There is another more common definition of symmetric subsets of a group as follows. If $S\sbs G$ where $G$ is a group, then $S$ is said to be symmetric if $s^{-1}\in S$ for all $s\in S$. This is \textit{not} the definition used here.
\end{rmk}
\begin{prop}\label{prop:permcupsm}
    Let $A, B$ be subsets of $S_n$.
    \begin{itemize}
        \item Suppose that $A\cap B = \varnothing$. If $A$ and $B$ are symmetric, then $A\cup B$ is symmetric. If $A$ and $B$ are Schur-positive, then $A\cup B$ is Schur-positive. 
        \item If $A\supseteq B$ and $A$ and $B$ are symmetric, then $A\sm B$ is symmetric. In particular, $S_n\sm B$ is symmetric if $B$ is symmetric.
    \end{itemize}
    \begin{proof}
        The first claim follows from the fact that $Q(A\cup B) = Q(A)+Q(B)$ for $A\cap B = \varnothing$ and that symmetric and Schur-positive functions are closed under addition. The second claim follows from the fact that $Q(A\sm B) = Q(A)- Q(B)$ if $A\supseteq B$ and that symmetric functions are closed under subtraction. The claim about $S_n\sm B$ being symmetric follows from the fact that $S_n$ is symmetric (see, e.g. \cite[Theorem 1.2]{HZP}).
    \end{proof}
\end{prop}
\begin{defn}
    A pattern set $\Pi\subseteq \bigcup\limits_{k = 1}^\infty S_k$ is said to be \textit{symmetrically avoided} if $S_n(\Pi)$ is symmetric for all $n$. It is said to be \textit{Schur-positively avoided} if $S_n(\Pi)$ is Schur-positive for all $n$.
\end{defn}
\begin{defn}\label{def:permcr}
    Given $w\in S_n$, its complement $w^c$ is defined as follows: $w^c(i) = n+ 1 - w(i)$. Its reversal $w^r$ is defined as follows  $w^r(i) = w(n+1-i)$. 
\end{defn}

Note that taking complement and reversal are both involutions. Given a set $S\subseteq S_n$, we write $S^c:= \{w^c: w\in S\}$ and $S^r:= \{w^r: w\in S\}$.
\begin{prop}[{\cite[Proposition 2.1 and 2.4]{HZP}}]\label{prop:crpat}
    If $S_n(\Pi)$ is symmetric \textup{(}respectively, Schur-positive\textup{)}, then $S_n(\Pi^c)$ and $S_n(\Pi^r)$ are symmetric \textup(respectively, Schur-positive\textup).
\end{prop}
The following corollary is implicitly proved in the proof of Proposition \ref{prop:crpat} in \cite{HZP}. 
\begin{cor}\label{cor:crset}
    If $S\subseteq S_n$ is symmetric \textup(respectively, Schur-positive\textup), then $S^c$ and $S^r$ are symmetric \textup(respectively, Schur-positive\textup).
    \begin{proof}
        Let $T = S_n \sm S$. Then $S_n(T) = S$ and $S_n(T^c) = S^c$ and $S_n(T^r) = S^r$, so the claim follows by Proposition \ref{prop:crpat}.
    \end{proof}
\end{cor}
\begin{cor}\label{cor:crpat}
    If $\Pi$ is symmetrically avoided \textup(respectively, Schur-positively avoided\textup), then $\Pi^c$ and $\Pi^r$ are symmetrically avoided \textup(respectively, Schur-positively avoided\textup).
    \begin{proof}
        This follows from the definition of symmetrically avoided, Schur-positively avoided, and Proposition \ref{prop:crpat}.
    \end{proof}
\end{cor}

\subsection{The Robinson--Schensted correspondence and its properties}
The Robinson--Schensted correspondence is an explicit bijection between permutations in $w\in S_n$ and pairs $(P, Q)$ of tableaux of the same shape of size $n$.
\begin{prop}\label{prop:rskprop}
    Suppose that $w\in S_n$ maps to $(P, Q)$ under the Robinson--Schensted correspondence. 
    \begin{itemize}
        \item If $w^c\mapsto (P', Q')$ under the Robinson--Schensted correspondence then $Q' = Q^t$.
        \item $\Des(w) = \Des(Q)$.
    \end{itemize}
\end{prop}

\section{Harmonic set systems}\label{sec:har}

\begin{defn}
    A \textit{set system} $H$ is a pair $H = (U, \Fcal),$ where $U$ is a set and $\Fcal = (A_1, \dots, A_m)$ is an ordered tuple of subsets of $U$. We call $U$ the universe set of $H$.
\end{defn}
\begin{defn}\label{def:symm2harset}
    Using the idea of Marmor \cite{M25}, given a set of permutations $S \subseteq S_n$, we can define an associated set system $\pa{S, \pa{A_1, \dots, A_{n-1}}}$ as follows.
\begin{equation}
    A_i := \{\pi\in S\mid i\in \Des(\pi)\} = \{\pi\in S\mid \pi(i) > \pi(i+1)\}
\end{equation}

We will denote this associated set system $A(S)$.
\end{defn}
\begin{rmk}
    Note that this slightly differs from the set system defined by Marmor \cite[Lemma 4.3]{M25}. The elements of the sets defined here are the permutations themselves, while in \cite{M25} the elements are indices instead (and thus rely on a choice of indexing the permutations in $S$.) Furthermore, we define $A_i$ to be the permutations that have a descent at $i$, rather than an ascent at $i$ which is what was done in \cite{M25}. Thus, our set system will be isomorphic to the complement of the one defined in \cite{M25} (we will define ``isomorphic'' and ``complement'' below). In view of \Cref{cor:crset}, \Cref{prop:comsetperm}, and \Cref{cor:comharset}, this is simply a matter of convention.
\end{rmk}
\begin{defn}[Isomorphic set systems]
    Two set systems $H_1 = (U,(A_1, \dots, A_m))$ and\\ $H_2  = (V,(B_1, \dots, B_m))$ are said to be \textit{isomorphic} if there exists a bijection $\sigma \colon U\to V$
    such that $\sigma(A_i) = B_i$ for all $i$. We write $H_1\cong H_2$ to denote that $H_1$ and $H_2$ are isomorphic.
\end{defn}
\begin{rmk}
    Note that we allow reordering of the elements, but not the sets. This is because the ordering of sets matters for Definition \ref{def:harmonicset} below.
\end{rmk}
\begin{defn}[Complement set systems]
    Given a set system $H =(U,\Fcal)$ with $\Fcal = (A_1, \dots, A_m)$, we define the complement of $H$ to be  $\ove{H} = (U, \ove{\Fcal})$ where 
    \begin{equation}
        \ove{\Fcal} := \pa{\ove{A_1}, \dots, \ove{A_m}},
    \end{equation}
    where $\ove{A_i} := U\sm A_i$.
\end{defn}
The following Proposition shows that this definition of complement behaves well with taking complement of permutation as in Definition \ref{def:permcr}.
\begin{prop}\label{prop:comsetperm}
    We have 
    \begin{equation}
        \ove{A(S)} \cong A(S^c).
    \end{equation}
    \begin{proof}
        Let $A(S)= (A_1, \dots, A_{n-1})$, and $A(S^c) = (B_1, \dots, B_{n-1})$. It suffices to show that $(S\sm A_i)^c = B_i$. Note that we have 
        \begin{align}
            (S\sm A_i)^c &= \{w\in S\mid w\notin A_i\}^c = \{w\in S\mid w(i) < w(i+1)\}^c\notag\\
            & = \{w\in S\mid w^c(i) > w^c(i+1)\}^c = \{w^c\in S\mid w(i) > w(i+1)\}\\
            &= \{w\in S^c\mid w(i)> w(i+1)\} = B_i. \qedhere\notag
        \end{align}
    \end{proof}
\end{prop}

\begin{defn}
    A set system $H = (U, (A_1, \dots, A_m))$ is said to be \textit{reduced} if 
    \begin{equation}
        \bigcap_{i = 1}^m A_i = \varnothing.
    \end{equation} It is said to be \textit{complete} if \begin{equation}
        \bigcup_{i = 1}^m A_i = U.
    \end{equation}
\end{defn}

\begin{prop}\label{prop:conid}
    A set $S\subseteq S_n$ contains $\delta_n$ if and only if $A(S) = (A_1, \dots, A_{n-1})$ is not reduced. Moreover, $S$ contains $\iota_n$ if and only if $A(S)$ is not complete.
    \begin{proof}
        If $\delta_n\in S$, then $\delta_n\in A_i$ for all $i$, so $A(S)$ is not reduced. Conversely, if $A(S)$ is not reduced, then there exists $\pi\in S$ such that $\pi\in A_i$ for all $i$, which implies $\pi(i)> \pi(i+1)$ for all $i$, so $\pi = \delta_n$.

        If $\iota_n\in S$ then $\iota_n\notin A_i$ for all $i$, so 
        \begin{equation}\bigcup\limits_{i = 1}^{n-1} A_i \subsetneq S.
        \end{equation}
        Conversely, if $A(S)$ is not complete, then there exists $\pi\in S$ such that $\pi\notin A_i$ for all $i$, which implies $\pi(i)< \pi(i+1)$ for all $i$, so $\pi = \iota_n$.
    \end{proof}
\end{prop}

\begin{defn}
    Given a set $I$ of positive integers, we define its \textit{run decomposition} as follows. The runs in $I$ are the maximal sequences of consecutive integers in $I$. The sizes of the runs sorted in nonincreasing order will make a partition of $|I|$, which we will call the run decomposition of $I$.
\end{defn}
\begin{ex}
    If $I = \{6, 2, 4,9, 1 , 5\}$, then the runs are $(1, 2), (4, 5, 6),(9)$. These have sizes of $2, 3, 1$ respectively, so the run decomposition of $I$ is the partition $(3, 2, 1)$.
\end{ex}
\begin{defn}\label{def:setinter}
    Given a set system $H = (U, (A_1, \dots, A_m)), $ and given $I_1, I_2\sbs [m]$, we define 
    \begin{equation}
        H_{I_1, I_2} := \bigcap_{i_1\in I_1}A_{i_1}\cap{\bigcap_{i_2\in I_2}\ove{A_{i_2}}},
    \end{equation}
    with the convention that the corresponding intersection is $U$ if $I_1$ or $I_2$ is empty. We write $H_I$ for $H_{I, \varnothing}.$ 
\end{defn}

\begin{defn}\label{def:harmonicset}
    A set system $H = \pa{U, \pa{A_1, \dots, A_m}}$ is said to be \textit{harmonic} if for any sets of indices $I$ and $J$ having the same run decomposition, we have $\abs*{H_I} = \abs*{H_J}.$
\end{defn}
The point of Definition \ref{def:harmonicset} is the following proposition.
\begin{prop}\label{prop:symm2harset}
    A set $S\subseteq S_n$ is symmetric if and only if $A(S)$ is harmonic.
\end{prop}

To prove Proposition \ref{prop:symm2harset}, we need the following definitions and propositions:
\begin{defn}[{\cite[Definition 4.1]{M25}}]
    Given a permutation $w\in S_n$ and a composition $\alpha = (\alpha_1, \dots, \alpha_k)$ of $n$, we say $w$ respects $\alpha$ if it is increasing along the segments of $\alpha$; that is, for any $1\le i\le k$ we have $w(b_{i-1}+1) < w(b_{i-1}+2) < \dots < w(b_{i}) $ where $b_i = a_1 +\dots +a_i$ (and so $b_0 = 0$). Equivalently, $\Des(w)\subseteq S_\alpha$. Given $S\subseteq S_n$ and $\alpha \vDash n$, we write 
    \begin{equation}
        S(\alpha):= \{w\in S\mid w\text{ respects }\alpha\}.
    \end{equation}
\end{defn}
\begin{defn}[{\cite[Definition 2.1]{M25}}]
    For two compositions $\alpha, \beta\vDash n$, we say $\alpha$ is equivalent to $\beta$ and write $\alpha\sim\beta$ if they are rearrangements of each other. Equivalently, the partition obtained by sorting them in descending order is the same.
\end{defn}
\begin{prop}[{\cite[Corollary 2.3]{M25}}]\label{prop:symmcond}
    A set $S\subseteq S_n$ is symmetric if and only if $|S(\alpha)| = |S(\beta)|$ for all $\alpha\sim \beta$.
\end{prop}
\begin{prop}\label{prop:rund2sim}
    Given $I\subseteq [n-1]$, let $I^c$ denote $[n-1]\sm I$. Then $I, J\subseteq [n-1]$ have the same run decomposition if and only if $\alpha_{I^c}\sim \alpha_{J^c}$.
    \begin{proof}
        Note that a run of length $t$ in $I$ gives a segment of length $t+1$ in $\alpha_{I^c}$: if $a+1, \dots, a+t$ is a run of length $t$ in $I$ then $a, a+t+1$ are in $I^c\cup\{0, n\}$, so this gives a segment of length $a+t-1 - a = t+1$ in the composition. Thus, if $\lambda = (\lambda_1, \dots, \lambda_k)$ is the run decomposition of $I$ then $(\lambda_1+1, \lambda_2+1, \dots, \lambda_k+1, 1^{n-|\lambda|-k})$ is the partition obtained by sorting $\alpha_{I^c}$. As a result, $I, J\subseteq [n-1]$ have the same run decomposition if and only if $\alpha_{I^c}\sim \alpha_{J^c}$.
    \end{proof}
\end{prop}
\begin{proof}[Proof of Proposition \ref{prop:symm2harset}]
    Let $A(S) = (S,(A_1, \dots, A_m))$. By \Cref{cor:crset}, $S$ is symmetric if and only if $S^c$ is symmetric. By Proposition \ref{prop:symmcond}, this is equivalent to $|S^c(\alpha)| = |S^c(\beta)|$ for all $\alpha\sim \beta$. 
    But note that 
    \begin{equation}
        w\in A(S)_I = \bigcap\limits_{i\in I}A_i
    \end{equation}
    is equivalent to $w(i) > w(i+1)$ for all $i\in I$, which is equivalent to $w^c(i) < w^c(i+1)$ for all $i\in I$, which is equivalent to $\Des(w^c)\subseteq I^c$, which is equivalent to $w^c\in S^c(\alpha_{I^c})$. So we have 
    \begin{equation}
        \abs*{A(S)_I} = \abs*{S^c(\alpha)}
    \end{equation}
    Thus, by Proposition \ref{prop:rund2sim}, $|S^c(\alpha)| = |S^c(\beta)|$ for all $\alpha\sim \beta$ if and only if 
    $\abs*{A(S)_I} = \abs*{A(S)_J}$
    for all $I, J\subseteq [n-1]$ with the same run decomposition.
\end{proof}

\begin{defn}\label{def:equivpairset}
    Two pairs of sets $(I_1, I_2)$ and $(J_1, J_2)$ with $I_1, I_2, J_1, J_2\subseteq [m]$ and $I_1\cap I_2 = J_1\cap J_2 = \varnothing$ are said to be equivalent if there exists a bijection $\sigma\colon [m]\to [m]$ such that 
    \begin{itemize}
        \item $\sigma(I_1) = J_1, \sigma(I_2) = J_2$.
        \item For $i$ and $j$ such that $i, j\in I_1\cup I_2$ we have $|\sigma(i)-\sigma(j)| = 1$ if and only if $|i-j| = 1$. 
    \end{itemize}
    We write $(I_1, I_2)\sim (J_1, J_2)$ to denote $ (I_1, I_2)$ and $(J_1, J_2)$ are equivalent.
\end{defn}

Informally speaking, two pairs of sets are equivalent if one can be obtained from the other by doing the following. We consider the consecutive runs of numbers in $I_1\cup J_1$, and move them around in blocks without splitting a block apart or combining blocks together. We are also allowed to flip a block backward. 

\begin{ex}
    $\textup(\{1,7,9\}, \{2,3,6\}\textup)$ is equivalent to $\textup(\{3, 5, 9\}, \{2,7,8\}\textup)$. Informally, we moved the $123$ block \textup(in the union of $\{1,7,9\}$ and $\{2,3,6\}$\textup) to the right so that it become $789$ and flipped it, we moved the $67$ block to the left so that it becomes $23$, and we moved the $9$ block to the left so that it becomes $5$. Formally, one can take the bijection $9,8,7,1,4,2,3,6,5,10$, which satisfies the conditions in Definition \ref{def:equivpairset}.
\end{ex}
\begin{rmk}\label{rmk:decompsim}
    $(I, \varnothing)\sim (J, \varnothing)$ if and only if $I$ and $J$ have the same run decomposition.
\end{rmk}

\begin{lm}\label{lm:doublehar}
    A set system $H = (U, (A_1, \dots, A_m))$ is harmonic if and only if for all $(I_1, I_2)\sim (J_1, J_2)$, we have $\abs*{H_{I_1, I_2}} = \abs*{H_{J_1, J_2}}$.
    \begin{proof}
        By Remark \ref{rmk:decompsim}, it suffices to show the ``only if'' direction. Fix $(I_1, I_2)\sim (J_1, J_2)$ and let $\sigma\colon [m]\to [m]$ be the bijection that satisfies the conditions in Definition \ref{def:equivpairset}. Then by the inclusion-exclusion principle we have 
        \begin{equation}
            \abs*{H_{I_1, I_2}} = \abs*{H_{I_1}\cap \bigcap_{i_2\in I_2}\ove{A_{i_2}}} = \sum_{T\sbs I_2}(-1)^{|T|}\abs*{H_{I_1}\cap \bigcap_{i_2\in T}A_{i_2}} = \sum_{T\sbs I_2}(-1)^{|T|}\abs*{H_{I_1\cup T}}
        \end{equation}
        Note that $(I_1\cup T, \varnothing)\sim (\sigma(I_1\cup   T), \sigma(\varnothing)) = (J_1\cup \sigma(T), \varnothing)$, so $I_1\cup T$ and $J_1\cup \sigma(T)$ has the same run decomposition. Thus, we have
        \begin{equation}
            \sum_{T\sbs I_2}(-1)^{|T|}\abs*{H_{I_1\cup T}} = \sum_{T\sbs I_2}(-1)^{|T|}\abs*{H_{J_1\cup \sigma(T)}} = \sum_{T\sbs J_2}(-1)^{|T|}\abs*{H_{J_1\cup T}},
        \end{equation}
        where the last equality is due to $\sigma(I_2) =J_2$. 
        But note that by similar reasoning as with $(I_1, I_2)$, we have 
        \begin{equation}
            \abs*{H_{J_1, J_2}} = \sum_{T\sbs J_2}(-1)^{|T|}\abs*{H_{J_1\cup T}}
        \end{equation}
        and thus $|H_{I_1, I_2}| = |H_{J_1, J_2}|$ as desired.
    \end{proof}
\end{lm}
\begin{cor}\label{cor:comharset}
    A set system $H$ is harmonic if and only if $\ove{H}$ is.
    \begin{proof}
        Suppose $H = (U, (A_1, \dots, A_m))$. Then $\ove{H} = (U, (B_1, \dots, B_m))$ where $B_i = \ove{A_i}$. Then we have \begin{equation}
            \ove{H}_{I_1, I_2}= \bigcap\limits_{i_1\in I_1}B_{i_1} \cap \bigcap\limits_{i_2\in I_2}\ove{B_{i_2}} = \bigcap\limits_{i_1\in I_1}B_{i_1} \cap \bigcap\limits_{i_2\in I_2}A_{i_2}= \bigcap\limits_{i_2\in I_2}A_{i_2}\cap \bigcap\limits_{i_1\in I_1}\ove{A_{i_1}} = H_{I_2, I_1}. 
        \end{equation}
        But note that $(I_1, I_2)\sim (J_1, J_2)$ if and only if $(I_2, I_1)\sim (J_2, J_1)$, since in Definition \ref{def:equivpairset}, the role of the indices $1$ and $2$ are symmetric.
    \end{proof}
\end{cor}

\section{Possible sizes of symmetric sets}\label{sec:sizeofsymm}
In this section, we prove Theorem \ref{thm:symsetlargsize} and Theorem \ref{thm:symsetsmallsize} which classify the possible sizes of a symmetric set without monotone elements.
\begin{rmk}\label{rmk:delmonotone} Any subset of $\{\iota_n, \delta_n\}$ is symmetric, so by Proposition \ref{prop:permcupsm}, $S\sbs S_n$ is symmetric if and only if $S\sm \{\iota_n, \delta_n\}$ is symmetric. Therefore, it suffices to ask for the possible sizes of a symmetric set without monotone elements, rather than asking for the possible sizes of a symmetric set.
\end{rmk}
\begin{rmk}\label{rmk:flipsize}
    If $S\sbs S_n$ is a symmetric set without monotone elements, then by Proposition \ref{prop:permcupsm}, $S_n\sm (S\cup \{\iota_n, \delta_n\})$ is a symmetric set without monotone elements. Therefore, to completely classify all the possible sizes of symmetric sets without monotone elements, it suffices to classify all the possible sizes up to $\frac{n!-2}2.$
\end{rmk}
We start by discussing \textit{Knuth-closed sets}, which are special examples of Schur-positive (and thus symmetric) sets.
\begin{defn}
    For each SYT $P$ of size $n$, the Knuth class corresponding to $P$, denoted $K(P)$, is the set of permutation $w\in S_n$ such that $w\mapsto (P, Q)$ under the Robinson--Schensted correspondence for some SYT $Q$.
\end{defn}
\begin{thm}[{\cite[Theorem 5.5]{gesselreutenauer}}]\label{thm:kcsp}
    Knuth classes are Schur-positive \textup(and thus symmetric\textup).
    \begin{proof}
        Consider the Knuth class of a SYT $P$ of size $n$ and let $\lambda = \sh(P)$. Given a permutation $w\in S_n$, write $(P(w), Q(w))$ for its image under the Robinson--Schensted correspondence. Recall from Proposition \ref{prop:rskprop} that $\Des(w) =\Des(Q(w))$. Then 
        \begin{equation}
            \sum_{w\in K(P)}F_{\Des(w)} = \sum_{w\in K(P)}F_{\Des(Q(w))} = \sum_{Q\in \SYT(\lambda)}F_{\Des(Q)} = s_{\lambda},
        \end{equation}
        where the last equality is by Theorem \ref{thm:slambda}.
    \end{proof}
\end{thm}
We say a set is \textit{Knuth-closed} if it is a union of Knuth classes. By Proposition \ref{prop:permcupsm} and Theorem \ref{thm:kcsp}, Knuth-closed sets are Schur-positive (and thus symmetric.) Given $\lambda\vdash n, $ let $f^\lambda$ denotes the number of SYT of shape $\lambda$. 
\begin{defn}
    Given $A, B\sbs \Z$, define $A+B := \{a+b\mid a\in A, b\in B\}.$ We write $nA$ for the sum of $n$ copies of $A$. We use the notation $[a,b)$ to denotes $[a, a+1, \dots, b-1).$
\end{defn}
The following proposition gives a good control on the possible sizes of Knuth-closed sets.
\begin{prop}
    There exists a Knuth-closed set of size $p$ without monotone elements if and only if 
    \begin{equation}
        p \in \sum_{\lambda \vdash n, \lambda \ne (1^n), \lambda \ne (n)}f^\lambda\{0, f^{\lambda}\}\label{eq:knuthcsize}
    \end{equation}
    \begin{proof}
        For each $\lambda\vdash n$ there are $f^\lambda$ possible $P$ with $\sh(P) = \lambda$, each of which corresponds to a Knuth class of size $f^\lambda$ since there are $f^\lambda$ choices of $Q.$ The two partitions $(1^n)$ and $(n)$ each correspond to a Knuth class containing only one element, which is $\iota_n$ or $\delta_n$.
        Thus, \eqref{eq:knuthcsize} holds if and only if $p$ can be written as the sum of sizes of distinct Knuth classes without monotone elements.
    \end{proof}
    \label{prop:knuthcsize}
\end{prop}
Our next goal is to show that the sumset on the right-hand side of \eqref{eq:knuthcsize} contains a large interval (Lemma \ref{lm:bigsymsize}).
\begin{thm}[Hook length formula]
    For $\lambda\vdash n$ we have 
    \begin{equation}
        f^\lambda = \frac{n!}{\prod\limits_{(i, j)\in \lambda}h(i, j)}
    \end{equation}
    where $h(i, j) = (\lambda_i-i) + (\lambda^t_j - j) +1$ is the hook length of the cell $(i, j).$\label{thm:hooklength}
\end{thm}
\begin{lm}
    Let $\lambda\vdash n$ and let $\mu$ be a partition obtained by moving a box of $\lambda$ to the first row. Then $f^{\mu}\ge \frac{f^\lambda}{n-1}.$
    \begin{proof}
        Note that $f^\mu\ge f^{\lambda\cap \mu}$ since given an SYT of shape $\lambda\cap \mu$ we can fill $\mu \sm \lambda\cap \mu$ with $n$ to obtain an SYT of $\mu.$ Note that $(n-1)f^{\lambda\cap \mu}\ge f^{\lambda}$ since $\lambda\sm \lambda\cap \mu$ must be filled with one of the $n-1$ values $2,\dots, n$ and the ordering of the remaining numbers is determined by an SYT of shape $\lambda\cap \mu.$ Thus $f^{\mu}\ge \frac{f^\lambda}{n-1}.$
    \end{proof}
    \label{lm:movetofirst}
\end{lm}
\begin{lm}
    Let $\lambda \vdash n$ such that $f^\lambda > 1.$ Then there exists $\mu\vdash n$ such that $f^\lambda > f^\mu \ge \frac{f^{\lambda}}{n-1}.$
    \begin{proof}
        We repeatedly move one box of $\lambda$ to the first row until we obtain a $\mu$ such that $f^{\mu}< f^\lambda$. This process halts, since if we move all boxes to the first row then $f^{\mu} = 1 < f^{\lambda}.$ Let $\mu'$ be the partition one step before halting. Then we have $f^{\mu'}\ge f^{\lambda}$, and thus $f^{\lambda} > f^\mu \ge \frac{f^{\mu'}}{n-1}\ge \frac{f^{\lambda}}{n-1},$ by Lemma \ref{lm:movetofirst}.
    \end{proof}
    \label{lm:divnm1}
\end{lm}
\begin{lm}\label{lm:squarelarger}
    Let $\lambda^{(1)}, \dots, \lambda^{(t)}$ denote the partitions of $n$ other than $(1^n), (n), (n-1,1), (n-2,1,1), (n-2,2)$, ordered such that $f^{\lambda^{(i)}}$ is nondecreasing. For all $i$ we have 
    \begin{equation}
        (n-1) +\sum_{0\le j < i}\pa{f^{\lambda^{(j)}}}^2 \ge f^{\lambda^{(i)}}.\label{eq:squarelarger}
    \end{equation}
    \begin{proof}
       The statement can be easily verified for $n\le 5$, so we assume $n\ge 6$. If $f^{\lambda^{(i)}}\ge (n-1)^2$ then by Lemma \ref{lm:divnm1} there is $\mu$ such that \begin{equation}f^{\lambda^{(i)}}>f^{\mu}\ge \frac{f^{\lambda^{(i)}}}{n-1}\ge n-1.\end{equation} This means $\mu \ne (1^n), (n),$ and if $\mu \in\{ (n-1, 1), (n-2, 1, 1), (n-2, 2)\}$ then we replace $\mu$ with it transpose $\mu^t$, so $\mu = \lambda^{(j)}$ for some $j < i.$ Then \begin{equation}(f^{\mu})^2 \ge \pa{\frac{f^{\lambda^{(i)}}}{n-1}}^2\ge f^{\lambda^{(i)}}\end{equation}
        so \eqref{eq:squarelarger} holds. If $n-1< f^{\lambda^{(i)}} < (n-1)^2$ then there is $j < i$ such that $\lambda^{(j)} = (n-1,1)^t$ and $(f^{\lambda^{(j)}})^2 = (n-1)^2\ge f^{\lambda^{(i)}}$ so \eqref{eq:squarelarger} holds. If $f^{\lambda^{(i)}} \le n-1$, \eqref{eq:squarelarger} clearly holds. Thus \eqref{eq:squarelarger} holds for all $i$.
    \end{proof}
\end{lm}
\begin{prop}\label{prop:addseg}
    If $b-a\ge c,$ we have $[a,b)+d\{0,c\}= [a,b+dc)$ for $d\ge 0$.
    \begin{proof}
        We prove by induction on $d$. The base case $d = 0$ is clear. Assume the statement is true for $d$; we show it is true for $d+1$. We have 
        \begin{align}
            &[a,b)+(d+1)\{0,c\}= [a,b]+d\{0,c\} + \{0, c\}= [a,b+dc)+\{0,c\}\\
            &= [a,b+dc) \cup [a+c, b+(d+1)c) = [a,b+(d+1)c).\qedhere\notag
        \end{align}
    \end{proof}
\end{prop}
\begin{prop}\label{prop:addsegpro}
    Let $a,b, c_1, \dots, c_k, d_1,\dots, d_k\ge 0$ be integers with $a < b.$
    If \begin{equation}b-a + \sum_{j < i}d_jc_j\ge c_i\end{equation} for $1\le i\le k,$ then \begin{equation}[a,b) + \sum_{i=1}^kd_i\{0, c_i\}= \left[a, b +  \sum_{i=1}^kd_ic_i\right).\end{equation}
    \begin{proof}
        This follows from Proposition \ref{prop:addseg} and induction on $k$.
    \end{proof}
\end{prop}
\begin{lm}\label{lm:startsum}
    For $n\ge 3$ one has
    \begin{align}
        &\frac{n(n-3)}2\left\{0,\frac{n(n-3)}2\right\}+ \floor*{\frac{n-3}2}\left\{0,n-1\right\} + (n-2)\left\{0, \frac{(n-1)(n-2)}2\right\}\notag\\
        &\supseteq \left[(n-2)\frac{n(n-3)}{2}, \frac{(n-1)(n-2)}2\cdot\frac{n(n-3)}2\right)\label{eq:startsum}.
    \end{align}
    \begin{proof}
        Start with any $y$ with $(n-2)\frac{n(n-3)}{2}\le y< \frac{(n-1)(n-2)}2\cdot\frac{n(n-3)}2.$ By the division algorithm one can write $y = c\frac{n(n-3)}2+d$ with $0 \le d < \frac{n(n-3)}2,$ and by the constraint on $y$ we have $n-2\le c< \frac{(n-1)(n-2)}2$ and so $n-2\le c \le \frac{n(n-3)}2.$ By the division algorithm, one can write $d = q(n-1)+r$ with $0\le r< n-1$ and \begin{equation}q < \frac{n(n-3)}{2(n-1)} <\frac{(n-1)(n-2)}{2(n-1)} = \frac{n-2}2\end{equation} and so $q\le \floor*{\frac{n-3}2}.$ Thus, we have 
        \begin{align}
            y &= c\frac{n(n-3)}2 + d = c\frac{n(n-3)}2 + q(n-1)+r \notag\\&= (c-r)\frac{n(n-3)}2+q(n-1)+r\frac{(n-1)(n-2)}2
        \end{align}
        with $0\le c-r\le \frac{n(n-3)}2$ and $q \le \floor*{\frac{n-3}2}$ and $r\le n-2.$ This shows \eqref{eq:startsum}.
    \end{proof}
\end{lm}
\begin{lm}
    For $n\ge 5$, for $(n-2)\frac{n(n-3)}{2}\le p\le \frac{n!-2}{2}$, there exists a Knuth-closed set of size $p$ without monotone elements.
    \begin{proof}
        Note that, using the hook length formula, we have $f^{(n-2, 1, 1)} = \binom{n-1}2 = \frac{(n-1)(n-2)}{2},$ $f^{(n-2, 2)} =  \frac{n(n-3)}2,$ and $f^{(n-1,1)} = n-1$.
        Now note that 
        \begin{align}
            &\frac{(n-1)(n-2)}2\cdot\frac{n(n-3)}2- (n-2)\frac{n(n-3)}{2} \notag\\&= \frac{n(n-3)(n-2)(n-3)}4\ge \frac{(n-1)(n-2)}2\ge n-1
        \end{align}
        for $n\ge 5.$ Thus, by Lemma \ref{lm:startsum} and Proposition \ref{prop:addsegpro} we have 
        \begin{align}\label{eq:startsumsq}
            &f^{(n-2, 2)}\{0,f^{(n-2, 2)}\}+f^{(n-1,1)}\{0,f^{(n-1,1)}\}\\
            &+ f^{(n-2,1,1)}\{0, f^{(n-2,1,1)}\}\supseteq \left[(n-2)\frac{n(n-3)}{2},T\right)\notag
        \end{align}
        where 
        \begin{align}
            T &:= \frac{(n-1)(n-2)n(n-3)}4+\pa{n-1-\floor*{\frac{n-3}2}}(n-1) \notag\\&+ \frac{(n-2)(n-3)}2\cdot \frac{(n-1)(n-2)}2\label{eq:calt}\\
            &=\frac{(n-2)(n-3)(n-1)^2}2 + \pa{n-1-\floor*{\frac{n-3}2}}(n-1).\notag
        \end{align}
        Note that
        we have 
        \begin{equation}
            T - (n-2)\frac{n(n-3)}2\ge \frac{(n-2)(n-3)}2\pa{(n-1)^2 - n}\ge (n-2)(n-3)\ge n-1,
        \end{equation}
        so by \eqref{eq:startsumsq}, \Cref{lm:squarelarger}, and \Cref{prop:addsegpro} we have 
        \begin{equation}
            \sum_{\lambda\vdash n, \lambda \ne (1^n), (n)}f^{\lambda}\{0,f^{\lambda}\}\sps \left[(n-2)\frac{n(n-3)}2,T + \hspace{-3mm}\sum_{\lambda\notin \{(1^n), (n), (n-1,1), (n-2,1,1), (n-2,2)\}}\hspace{-3mm}(f^{\lambda})^2\right)
        \end{equation}
        Note that we have 
        \begin{equation}
            \sum_{\lambda\vdash n}(f^{\lambda})^2 = n!
        \end{equation}
        by the Robinson--Schensted correspondence, so 
        \begin{align}
            &T + \sum_{\lambda\notin \{(1^n), (n), (n-1,1), (n-2,1,1), (n-2,2)\}}(f^{\lambda})^2\notag\\ 
            &=  T + n! - \sum_{\lambda\in \{(1^n), (n), (n-1,1), (n-2,1,1), (n-2,2)\}}(f^{\lambda})^2\notag\\
            &= \frac{(n-2)(n-3)(n-1)^2}2 + \pa{n-1-\floor*{\frac{n-3}2}}(n-1) + n! \notag\\&- 2 - \pa{\frac{n(n-3)}2}^2 -\pa{\frac{(n-1)(n-2)}2}^2 - (n-1)^2\\
            &= n!-\frac{n^3}2 +3n^2-\frac{11n}2 - \floor*{\frac{n-3}2}(n-1)\ge n!-\frac{n^3}2 +3n^2-\frac{11n}2 - \frac{n-3}2(n-1)\notag\\
            &= n! - \frac{n^3-5n^2+7n+3}2 = n! -\frac{n(n-1)(n-2)-(2n+1)(n-3)}2\ge \frac{n!}2.\notag
        \end{align}
        By Proposition \ref{prop:knuthcsize}, we are done.
    \end{proof}\label{lm:bigsymsize}
\end{lm}
We will now construct symmetric sets of size smaller than $(n-2)\frac{n(n-3)}{2}$. For this, it will be helpful to have the following construction of a symmetric set of size $n$.
\begin{prop}\label{prop:symmsizen}
    A set of permutations $\{\pi_1, \dots, \pi_n\}$ with $\Des(\pi_i) = \{i-1, i\}\cap [n-1]$ for all $i$ or with $\Des(\pi_i) = [n-1]\sm \{i-1, i\}$ for all $i$ is symmetric.
    \begin{proof}
        By taking complement, without loss of generality assume $\Des(\pi_i) = \{i-1, i\}\cap [n-1]$ for all $i$.
        Note that $A(S) = (S, (A_1, \dots, A_{n-1}))$ where $A_i = \{\pi_i, \pi_{i+1}\}.$ We prove that $A(S)$ is harmonic. Note that any two nonconsecutive sets have empty intersection, and thus any intersection of at least three sets is empty, so to prove harmonicity it suffices to show that all the sets are of the same size and any intersection of the form $A_i\cap A_{i+1}$ has the same size. We can check that $|A_i| = 2$ for all $i$ and $|A_i\cap A_{i+1}| = 1$ for all $i$.
    \end{proof}
\end{prop}
The next lemma counts the number of permutations that are relevant to the construction in \Cref{prop:symmsizen}.

\begin{lm}\label{lm:countper}
    For any $1\le i\le n,$ there are at least $n-1$ permutations with descent set exactly $\{i-1, i\}\cap [n-1]$, and at least $n-1$ permutations with descent set exactly $[n-1]\sm \{i-1, i\}.$ For any $1\le i\le n-1$ there are at least $n-1$ permutations with descent set exactly $\{i\}$, and at least $n-1$ permutations with descent set exactly $[n-1]\sm\{i\}$.
    \begin{proof}
        By taking complement, it suffices to show that for $1\le i\le n,$ there are at least $n-1$ permutations with descent set exactly $\{i-1, i\}\cap [n-1],$ and for $1\le i\le n-1$ there are at least $n-1$ permutations with descent set exactly $\{i\}.$

        To construct a permutation whose descent set is a subset of $\{i\}$, it suffices to choose the first $i$ letters. Thus, the number of permutation whose descent set is a subset of $\{i\}$ is $\binom ni$, and so there are $\binom ni -1\ge n-1$ permutations whose descent set is exactly $\{i\},$ since $\iota_n$ is the only permutation with empty descent set.

        For $1 < i < n,$ to construct a permutation whose descent set is a subset of $\{i-1,i\}$, it suffices to choose the first $i-1$ letter, and then the $i$-th letter. This can be done in $\binom n{i-1}(n-i+1)$ ways. Thus, the number of permutation with descent set exactly $\{i-1,i\}$ is 
        \begin{align}
            &\binom n{i-1}(n-i+1) - \binom{n}{i-1}-\binom ni + 1 = \binom{n}{i-1}\pa{n-i-\frac{n-i+1}{i}}+1\notag\\
            &=\binom n{i-1}\frac{(n-i)(i-1)-1}{i} + 1\ge n-1. \qedhere
        \end{align}

    \end{proof}
\end{lm}
\begin{lm}\label{lm:bounddes}
    Every permutation of a Knuth class of a tableau of shape $\lambda$ has at most $n-\lambda_1$ descents.
    \begin{proof}
        By \Cref{prop:rskprop}, it suffices to show that a tableau of shape $\lambda$ cannot have more than $n-\lambda_1$ descents. Note that if $i$ is a descent then $i+1$ cannot be on the first row, and there are only $n-\lambda_1$ boxes not in the first row.
    \end{proof}
\end{lm}
We now use the construction in Proposition \ref{prop:symmsizen} to construct symmetric sets of smaller sizes. The idea is to write the desired size $p$ as a linear combination of $n-1$, $n$, $\frac{n(n-3)}2$, then use the construction in \Cref{prop:symmsizen} together with Knuth classes of size $n-1, \frac{n(n-3)}2.$ Unlike the proof of \Cref{lm:medsymsize}, which just take unions of disjoint Knuth classes, one needs to be careful to make the construction from Proposition \ref{prop:symmsizen} disjoint from the Knuth classes.
\begin{lm}\label{lm:medsymsize}
    For $n\ge 5$ and $(n-1)(n-3)\le p < (n-2)\frac{n(n-3)}{2},$ there is a symmetric set of size $p$ without monotone elements.
    \begin{proof}
        We first consider $n\ge 6.$
        Note that we have \begin{equation}2(n-1)^2 - (n-1)(n-3)\ge \frac{(n-1)(n-2)}2\ge \frac{n(n-3)}2\end{equation}
        so by Proposition \ref{prop:addseg} we have \begin{equation}\left[(n-1)(n-3), 2(n-1)^2\right) + (n-3)\{0, \frac{n(n-3)}2\}\sps \left[(n-1)(n-3), (n-2)\frac{n(n-3)}2\right).\end{equation}
        Thus, we can write $p = a + t\frac{n(n-3)}2$ for some $0\le t\le n-3,$ and $(n-1)(n-3)\le a < 2(n-1)^2.$ Now by the division algorithm, we can write $a = q(n-1)+r$ where $0\le r < n-1$, and $n-3\le q < 2(n-1)$. If $q \ge r$ we  write $a = (q-r)(n-1)+rn$. The only case when $q < r$ is when $q = n-3, r = n-2$ and thus $a = n(n-3)$, we write $a = 2\frac{n(n-3)}2$ instead in this case. Thus, in both case we can write 
        \begin{equation}
            p = u(n-1)+vn + d\frac{n(n-3)}2
        \end{equation}
        where $0\le d\le n-1$ (we may have to add at most $2$ to $t$) and $0\le u + v\le q < 2(n-1)$ and $0\le v\le r < n-1.$

        Let $y = \min(u, n-1)$. Take $d$ Knuth classes of shape $(n-2,2)$, $y$ Knuth classes of shape $(n-1, 1)$ and $u-y$ Knuth classes of shape $(n-1,1)^t$. Now for $n\ge 6$ we claim that for each $1\le i\le n$ there are still at least $v$ permutations of descent set exactly $[n-1]\sm\{i-1,i\}$ that has not been taken. Indeed, these permutations has at least $n-3\ge 3$ descents, so by \Cref{lm:bounddes} the only Knuth classes among the above that could have taken such permutations is the $u-y$ Knuth classes of shape $(n-1, 1)^t$. If $u-y = 0$ then the number of  permutations of descent set exactly $[n-1]\sm\{i-1,i\}$ remaining is at least $n-1$ by \Cref{lm:countper}, and $n-1\ge v$ so the claim is true. Now if $u-y > 0$ then $y = n-1$ and thus \begin{equation}u-y + v < 2(n-1) -(n-1) < n-1.\end{equation} Note that the shape $(n-1,1)^t$ has $n-1$ SYT formed by placing $i+1$ in the second column and the rest in the first column, which has descent set $[n-1]\sm \{i\}$, for $1\le i< n$. Thus, by \Cref{prop:rskprop}, for each $i$, each Knuth class of shape $(n-1,1)^t$ contains at most one permutation of descent set exactly $[n-1]\sm \{i-1 , i\}.$ Since $(u-y)+v < n-1$, there are still at least $v$ permutations of descent set exactly $[n-1]\sm \{i-1 , i\}.$

        Since there are at least $v$ permutations of descent set exactly $[n-1]\sm \{i-1 , i\}$ for each $i$, we can form $v$ sets where each set has the form $\{\pi_1, \dots, \pi_n\}$ with $\Des(\pi_i) = [n-1]\sm \{i-1, i\},$ which is symmetric by Proposition \ref{prop:symmsizen}. This together with the Knuth classes above form a symmetric set of size exactly $p.$

        We proved the statement for $n\ge 6.$ For $n = 5$ note that $f^{(4,1)} = 4, f^{(3,2)} = 5, f^{(3,1,1)} = 6$ and one can verify that $4\{0,4\}+ 5\{0,5\}+6\{0,6\}\sps [8,15)$, so we are done by \Cref{prop:knuthcsize}.
    \end{proof}
\end{lm}
    \begin{proof}[Proof of Theorem \ref{thm:symsetlargsize}]
        The statement is true for $n\ge 5$ by Lemma \ref{lm:bigsymsize} and Lemma \ref{lm:medsymsize}. For $n = 4$, note that $f^{(3,1)}  = f^{(1,3)} = 3, f^{(2,2)} = 2,$ and one can verify that $2\{0,2\} + 3\{0,3\}+3\{0,3\}\sps [3, 12),$ so by Proposition \ref{prop:knuthcsize}, we are done.
    \end{proof}
We now classify the possible sizes of a symmetric set without monotone elements below $(n-1)(n-3).$ We first show the ``if'' direction of \Cref{thm:symsetsmallsize}.
\begin{lm}\label{lm:symmverysmall}
For $p < (n-1)(n-3)$, there exists a symmetric set of size $p$ without monotone elements if one of the following hold. 
\begin{enumerate}[label=\textup(\arabic*\textup)]
        \item $p = qn-r$ for some $0\le r \le q < n-3,$
        \item $n$ is even and $p = \frac{n(n-3)}2+c(n-1)$ for some $c\ge 0,$
        \item $n$ is odd and $p = \binom{n-1}2+cn$ for some $c\ge 0.$
    \end{enumerate}
\begin{proof}
    The lemma can be easily verified for $n\le 5$, so we assume $n\ge 6$.
    If (1) holds then $p = (q-r)n + r(n-1).$ By \Cref{lm:countper}, we can form $q-r\le n-1$ sets where each set has the form $\{\pi_1, \dots, \pi_n\}$ with $\Des(\pi_i) = [n-1]\sm \{i-1, i\},$ which is symmetric by Proposition \ref{prop:symmsizen}. Take $r$ Knuth classes of shape $(n-1,1).$ By \Cref{lm:bounddes}, these permutations have only one descent, and thus are disjoint from the $q-r$ sets above. Thus, by taking their union we have a symmetric set of size $p.$

    If (2) holds, then note that $c < (n-3)$ and thus we can take $c$ Knuth classes of shape $(n-1,1)$ together with a Knuth class of shape $(n-2,2)$. Note that $f^{(n-2,2)} = \frac{n(n-3)}2,$ so this is a symmetric set of size exactly $p.$
        
    If (3) holds, then note that with $\lambda = (n-2,1,1)$, we have $f^\lambda = \binom{n-1}2$ using the hook length formula. By \Cref{lm:bounddes}, permutations in a Knuth class of shape $(n-2,1,1)$ has at most two descents. Note that $c < n-3$ since $p < (n-1)(n-3).$ By \Cref{lm:countper}, we can form $c\le n-1$ sets where each set has the form $\{\pi_1, \dots, \pi_n\}$ with $\Des(\pi_i) = [n-1]\sm \{i-1, i\}.$ Then we can take a Knuth class of shape $(n-2,1,1)$ which is disjoint with the above $c$ sets since the permutations involved have at least $n-3$ descents. Thus, by taking their union we have a symmetric set of size $p.$
\end{proof}
\end{lm}

For the ``only if'' direction of \Cref{thm:symsetsmallsize}, given a symmetric set $S$, by considering the associated set system $A(S) := (S, (A_1, \dots, A_{n-1}))$ (recall \Cref{def:symm2harset}), the problem roughly translates to the problem of classifying the possible sizes of $|U|$ if $H = (U, (A_1, \dots, A_m))$ is a complete and reduced harmonic set system and $|U| < m(m-2)$ (note that $m = n-1$ is the relevant case.) This is done through a series of propositions and lemmas. 
\begin{prop}\label{prop:6nonconsect}
    Let $H = (U, (A_1, \dots, A_m))$ be a harmonic set system such that $|U| < m(m-2)$ and $m\ge 51.$ Then $H_{\{1,3,5\}, \{7,9,11\}} = \varnothing.$
    \begin{proof}
        Suppose $a\in H_{\{1,3,5\}, \{7,9,11\}}.$ Let $T$ be the set of odd indices $i$ such that $a\in A_i$, and let $O$ be the set of all odd indices. Then $a\in H_{T, O\sm T}.$ Then for any bijection $\sigma\colon O\to O$ we have $(T, O\sm T)\sim (\sigma(T), \sigma(O\sm T))$ (recall Definition \ref{def:equivpairset}.) Thus, for any $T'\sbs O$ such that $|T'| =|T|$ we have $(T, O\sm T)\sim (T', O\sm T'),$ so by Lemma \ref{lm:doublehar} we have $|H_{T', O\sm T'}| = |H_{T, O\sm T}|\ge 1.$

        Now for any $T_1, T_2\sbs O$ with $|T_1|  = |T_2| = |T|$ and $T_1\ne T_2$, we claim that $H_{T_1, O\sm T_1}\cap H_{T_2, O\sm T_2} = \varnothing$. Indeed, since $T_1\ne T_2$ but $|T_1| = |T_2|$, there is $t\in T_1\sm T_2$, and so $t\in O\sm T_2$. Then 
        \begin{equation}
            H_{T_1, O\sm T_1}\cap H_{T_2, O\sm T_2} \sbs A_{t}\cap \ove{A_{t}} = \varnothing.
        \end{equation}
        Thus, we have 
        \begin{equation}
            m(m-2) > |U| \ge \sum_{\substack{T'\sbs O\\|T'| = |T|}}|H_{T', O\sm T'}|\ge \sum_{\substack{T'\sbs O\\|T'| = |T|}}1 = \binom{|O|}{|T|}\ge \binom{\ceil*{\frac m2}}3\label{eq:halfchoose3}
        \end{equation}
        where the last inequality follows from the fact that $3\le |T|\le |O|-3$ and that $|O| = \ceil*{\frac m2}.$ One can verify that \eqref{eq:halfchoose3} fails when $m\ge 51.$
    \end{proof}
\end{prop}
\begin{defn}
    A set or tuple of integers is said to be nonconsecutive if no two distinct elements differ by exactly $1$.
\end{defn}
\begin{prop}\label{prop:5to3}
    A set of $5$ integers contains a nonconsecutive subset of size $3$.
    \begin{proof}
        Order the $5$ integers in ascending order and take the $1$st, $3$rd and $5$th elements.
    \end{proof}
\end{prop}
\begin{prop}\label{prop:3nonconsec}
    Let $m\ge 14$ and let $T\sbs [m]$. Suppose that both $T$ and $[m]\sm T$ each contain $3$ nonconsecutive integers.  Then one can find $3$ elements from $T$ and $3$ elements from $[m]\sm T$ such that all $6$ of them are nonconsecutive.
    \begin{proof}
        Consider the graph where the vertices are nonconsecutive subsets of size $6$ of $[m],$ and there is an edge between $A$ and $B$ if $B = A\sm \{a\}\cup \{b\}$ for some $a, b$---that is, if one can be obtained from another by changing only one element. We claim this graph is connected. Indeed, given $A = \{a_1 < \dots < a_6\}$, define $A_i := \{2j-1\mid j \le i\}\cup \{a_j\mid j > i\}.$ Then $A_0 = A$ and there is an edge between $A_i$ and $A_{i+1}$ and $A_6 = \{1,3,5,7,9,11\}.$ So any vertex is connected to $\{1,3,5,7,9,11\}$ which shows that the graph is connected.

        Now note that any nonconsecutive subset $A$ of size $3$ of $[m]$ can be extended to a nonconsecutive subset of size $6$ of $[m]$. Indeed, the set \begin{equation}[m]\sm\pa{A + \{-1,0,1\}}\end{equation}
        has at least $14-3\cdot3 = 5$ elements and thus has a nonconsecutive subset $B$ of size $3$ by Proposition \ref{prop:5to3} and one can take $A\cup B$.
         
        Now label each vertex of the above graph by the size of intersection with $T$.  Since $T$ has a nonconsecutive subset of size $3$ and we can extend this to a nonconsecutive subset of size $6$, there is a vertex with color at least $3$. Since $[m]\sm T$ has a nonconsecutive subset of size $3$ and we can extend this to a nonconsecutive subset of size $6$, there is a vertex with color at most $3$. There is a path between these two vertex, and along each edge the label change by at most $1$, so there is a vertex with label exactly $3$ as desired.
    \end{proof}
\end{prop}
\begin{prop}\label{prop:31or30}
    Let $H = (U, (A_1, \dots, A_m))$ be a harmonic set system such that $|U| < m(m-2)$ and $m\ge 51.$ Then for any two nonconsecutive sets $I, J$ of size $3$, we have $H_{I, J} = \varnothing.$
    \begin{proof}
        If $I\cap J \ne\varnothing$ then the conclusion is trivial. Otherwise, suppose $a\in H_{I, J}$. Let $T$ be the set of indices $i$ such that $a\in A_i$. Then $I\sbs T$ while $J\sbs [m]\sm T$, by Proposition \ref{prop:3nonconsec} there is a nonconsecutive $S$ of size $6$ such that $|S\cap T| = 3.$ Then we have $a\in H_{S\cap T, S\sm T}.$ Since the elements of $S$ are nonconsecutive, we have $(S\cap T, S\sm T)\sim (\{1,3,5\}, \{7,9,11\}).$ Thus, $H_{\{1,3,5\}, \{7,9,11\}}\ne\varnothing$ by Lemma \ref{lm:doublehar}, which contradicts Proposition \ref{prop:6nonconsect}.
    \end{proof}
\end{prop}
\begin{lm}\label{lm:split}
    Let $H = (U, (A_1, \dots, A_m))$ be a harmonic set system such that $|U| < m(m-2)$ and $m\ge 51.$ Then one can partition $U = U^{(1)}\sqcup U^{(2)}$ such that the induced systems $H^{(i)} := (U^{(i)}, (A_1\cap U^{(i)}, \dots, A_m\cap U^{(i)}))$ are harmonic and satisfies $H^{(1)}_{I} = H^{(2)}_{\varnothing, I} =\varnothing$ for any nonconsecutive $I$ of size $3.$ 
    \begin{proof}
        Let $\Hc$ denote the collection of all nonconsecutive subsets of indices of size $3$. Let  
        \begin{equation}
            U^{(1)} := \bigcup\limits_{I\in \Hc}H_{\varnothing,I}
        \end{equation}
        and 
        \begin{equation}
            U^{(2)} := \bigcup\limits_{I\in \Hc}H_{I}.
        \end{equation}
        Then for any $J\in \Hc$ we have, by \Cref{prop:31or30},
        \begin{equation}
            H^{(1)}_{J} := \bigcap_{j\in J}\pa{A_j\cap U^{(1)}} = H_J\cap U^{(1)} = H_J\cap \bigcup\limits_{I\in \Hc}H_{\varnothing,I} = \bigcup\limits_{I\in \Hc}\pa{H_J\cap H_{\varnothing,I}} = \bigcup\limits_{I\in \Hc}H_{J,I} = \varnothing
        \end{equation}
        and 
        \begin{equation}
            H^{(2)}_{\varnothing,J} := \bigcap_{j\in J}\pa{\ove{A_2}\cap U^{(2)}} = H_{\varnothing,J}\cap U^{(1)} = H_{\varnothing,J}\cap \bigcup\limits_{I\in \Hc}H_{I} = \bigcup\limits_{I\in \Hc}\pa{H_{\varnothing, J}\cap H_{I}} = \bigcup\limits_{I\in \Hc}H_{I,J} = \varnothing.
        \end{equation}
        We also have 
        \begin{equation}
            U^{(1)}\cap U^{(2)} = \bigcup\limits_{J\in \Hc}H_{\varnothing,J}\cap \bigcup\limits_{I\in \Hc}H_{I} = \bigcup_{I, J\in \Hc}H_{\varnothing,J}\cap H_{I} = \bigcup_{I, J\in \Hc}H_{I,J} = \varnothing,
        \end{equation}
        and we now show $U_1\cup U_2 = U.$ Fix $u\in U$, then let $T$ be the subset of indices $i$ such that $u\in A_i$. Then if $|T|\ge 5$, choose $A\sbs T$ such that $A\in \Hc$ (Proposition \ref{prop:5to3}). Then $u\in H_A\sbs U^{(1)}.$ Otherwise, we have $|[m]\sm T|\ge 5$, choose $A\sbs [m]\sm T$ such that $A\in \Hc$. We have $u\in H_{\varnothing,A}\sbs U^{(2)}.$

        It remains to show that $H^{(1)}, H^{(2)}$ are harmonic. We will prove this for $H^{(1)}$ as the proof for $H^{(2)}$ is analogous. Since $H^{(1)}_I =\varnothing$ for any nonconsecutive $I$ of size $3$, it suffices to show that $\abs*{H^{(1)}_I} =\abs*{H^{(1)}_J}$ for  nonempty $I\sim J$ and $I, J$ not containing $3$ nonconsecutive integers. Fix such a pair $I$ and $J$. Note that 
        \begin{equation}
            [m]\sm\pa{ (I\cup J)+\{-1,0,1\}}
        \end{equation}
        has at least $51-6\cdot 3\ge 9$ elements, and thus has $5$ nonconsecutive integers $T$. We have
        \begin{equation}
            \abs*{H_I^{(1)}} = \sum_{K\sbs T}\abs*{H_{I\cup K, T\sm K}^{(1)}}.
        \end{equation}
        Note that for each $K$, either $I\cup K$ or $T\sm K$ has $3$ nonconsecutive integers. If $I\cup K$ has $3$ nonconsecutive integers then $H_{I\cup K, T\sm K}^{(1)} = \varnothing.$ If $T\sm K$ has $3$ nonconsecutive integers then $\abs*{H_{I\cup K, T\sm K}^{(2)}} = 0$ so $\abs*{H_{I\cup K, T\sm K}^{(1)}} = \abs*{H_{I\cup K, T\sm K}}.$ Thus, letting $\Hc$ denotes the collection of $K\sbs T$ such that $T\sm K$ has $3$ nonconsecutive integers, by using a similar argument with $J$ we have 
        \begin{equation}
            \abs*{H_I^{(1)}} = \sum_{K\in \Hc}\abs*{H_{I\cup K, T\sm K}} = \sum_{K\in \Hc}\abs*{H_{J\cup K, T\sm K}} = \abs*{H_J^{(1)}},
        \end{equation}
        as desired.
    \end{proof}
\end{lm}
\begin{lm}\label{lm:calculations}
    Let $H = (U, (A_1, \dots, A_m))$ be a complete harmonic set system such that $|U|< m(m-2)$ and $m\ge 51$ and $H_I = \varnothing$ for all nonconsecutive $I$ of size $3$. Then $|U|$ is a nonnegative integer linear combination of $m, m+1, \binom m2, \frac{(m+1)(m-2)}{2}.$
    \begin{proof}
       
Let $H_\lambda$ denote $|H_I|$ for any set $I$ with run decomposition $\lambda.$ Note that since $H_I = \varnothing$ for all nonconsecutive $I$ of size $3$, we have $H_\lambda = 0$ unless \begin{equation}\lambda\in\{\varnothing, (1), (2), (1,1), (3), (2,1), (4), (2,2)\}.\end{equation}
Let $\Hc_\lambda$ denote the collection of all subsets of indices that have run decomposition $\lambda.$
For any $\lambda$ we have   
\begin{equation}
    \abs*{U} \ge \abs*{\bigcup_{I\in \Hc_{\lambda}}H_{I}} = \sum_{\varnothing \subsetneq P\sbs \Hc_\lambda}(-1)^{|P|-1}H_{\bigcup\limits_{I\in P}I}. 
    \label{eq:hlambda}
\end{equation}

Applying \eqref{eq:hlambda} with $\lambda = (2, 1)$ and $ \lambda = (1,1)$ we have
\begin{align}
    |U|&\ge (m-2)(m-3)H_{(2,1)}-3\binom{m-3}2H_{(2,2)}-(m-3)H_{(4)},\label{eq:dbound}\\
    |U| &\ge  \binom{m-1}2H_{(1,1)}-(m-2)(m-3)H_{(2,1)}+ \binom{m-3}2H_{(2,2)}\label{eq:fbound},
\end{align}
where we skipped some elementary counting of how many times certain sets appear in the right-hand side of \eqref{eq:hlambda}. For the rest of this proof, we will omit similar elementary computations.

By counting the number of elements that appears in $4$ sets we have  
\begin{equation}
    \abs{U} \ge (m-3)H_{(4)} + \binom{m-3}2H_{(2,2)}.\label{eq:fourbound}
\end{equation}

For $I$ with run decomposition $\lambda$ we have 
\begin{align}
    H_\lambda &= \abs*{H_I}  = H_{I, [m]\sm I} +\abs*{\bigcup_{j\in [m]\sm I}H_{I\cup\{j\}} } \ge  \sum_{\varnothing \subsetneq P\sbs [m]\sm I} (-1)^{|P|-1}\abs*{H_{I\cup P}}, \label{eq:ineqi}
\end{align}
and if $I = \varnothing$ then equality holds by completeness.

Applying \eqref{eq:ineqi} with $I = \{1,2,4\}, I = \{1,2,5\}, I = \{8,9\}, I = \{5,9\}, I = \{5,8\}, I = \{1\}, I = \{3\}, I = \varnothing$ respectively we get 
\begin{align}
    H_{(2,1)}&\ge H_{(2,2)}+H_{(4)},\label{eq:dab}\\
    H_{(2,1)}&\ge 2H_{(2,2)},\label{eq:da}\\
    H_{(2)}&\ge 2H_{(3)} + (m-4)H_{(2,1)} - 3H_{(4)} -(m-6)H_{(2,2)},\label{eq:evmid}\\
    H_{(1,1)}&\ge 4H_{(2,1)}-4H_{(2,2)},\label{eq:fda}\\
    H_{(1,1)}&\ge 4H_{(2,1)}-3H_{(2,2)}-H_{(4)},\label{eq:fab}\\
    H_{(1)}&\ge H_{(2)} + (m-2)H_{(1,1)} - H_{(3)} - (2m-6)H_{(2,1)}+H_{(4)}+(m-4)H_{(2,2)},\label{eq:gfront}\\
    H_{(1)}&\ge 2H_{(2)} + (m-3)H_{(1,1)} - 3H_{(3)} - (3m-13)H_{(2,1)}+3H_{(4)}+(2m-11)H_{(2,2)}\label{eq:gthree},\\
    \abs{U} &= mH_{(1)} - \pa{(m-1)H_{(2)} + \binom{m-1}2H_{(1,1)}} \notag\\
    &+ \pa{(m-2)H_{(3)} + (m-2)(m-3)H_{(2,1)}} - \pa{(m-3)H_{(4)} + \binom{m-3}2H_{(2,2)}}.\label{eq:onebound}
\end{align} 
For the rest of the proof, we will use \eqref{eq:dbound}--\eqref{eq:onebound} to show that $|U|$ is of the desired form.
By \eqref{eq:fourbound} 
we have 
\begin{align}
    m(m-2) & > |U|\ge  (m-3)H_{(4)} + \binom{m-3}2H_{(2,2)}, 
\end{align}
so $H_{(2,2)} < 3.$ We now split into cases.
\begin{casesp}
\item $H_{(2,2)} = 0.$
\begin{casesp}
    \item $H_{(2,1)} = 0.$ By \eqref{eq:dab} we have $H_{(4)} = 0.$ By \eqref{eq:evmid} we have $H_{(2)}\ge 2H_{(3)}.$ 
Using \eqref{eq:onebound} we have 
\begin{align}
    |U| &=mH_{(1)} - \pa{(m-1)H_{(2)} + \binom{m-1}2H_{(1,1)}} + (m-2)H_{(3)} \notag\\
    &= mg - (m-1)H_{(2)} - \frac{(m-1)(m-2)}2H_{(1,1)} + (m-2)H_{(3)}\notag\\
    &= m(H_{(1)} - (H_{(2)}+(m-2)H_{(1,1)} - H_{(3)})) + H_{(2)} +\frac{(m+1)(m-2)}2H_{(1,1)} -2H_{(3)}\label{eq:uwrite}\\
    &= m(H_{(1)} - (H_{(2)}+(m-2)H_{(1,1)} - H_{(3)}))  \notag\\
    &+\frac{(m+1)(m-2)}2(H_{(1,1)}-(H_{(2)}-2H_{(3)}))+ \binom{m}2(H_{(2)}-2H_{(3)}),\notag
\end{align}
where $H_{(2)}-2H_{(3)}\ge 0$, and $H_{(1)} - (H_{(2)}+(m-2)H_{(1,1)} - H_{(3)})\ge 0$ by \eqref{eq:gfront}. Thus if $H_{(1,1)}\ge H_{(2)}-2H_{(3)}$ then this is a nonnegative integer linear combination of $m, \frac{(m+1)(m-2)}2, \binom{m}2.$ If $H_{(1,1)} < H_{(2)} - 2H_{(3)}$ then by \eqref{eq:uwrite} we have 
\begin{align}
    |U| 
    &= m(H_{(1)} - (H_{(2)}+(m-2)H_{(1,1)} - H_{(3)})) + H_{(2)} +\frac{(m+1)(m-2)}2H_{(1,1)} -2H_{(3)}\notag\\
    &= m(H_{(1)} - (H_{(2)}+(m-2)H_{(1,1)} - H_{(3)}) - (H_{(2)}-2H_{(3)}-H_{(1,1)})) \notag\\
    &+ (m+1)(H_{(2)}-2H_{(3)}-H_{(1,1)}) +\binom{m}2H_{(1,1)}\label{eq:uwritenew}\\ 
    &= m(H_{(1)} - (2H_{(2)}+(m-3)H_{(1,1)} - 3H_{(3)})) + (m+1)(H_{(2)}-2H_{(3)}-H_{(1,1)}) \notag\\
    &+\binom{m}2H_{(1,1)}\notag 
\end{align}
where $H_{(1)} - (2H_{(2)}+(m-3)H_{(1,1)} - 3H_{(3)})\ge 0$ by \eqref{eq:gthree}.
    \item $H_{(2,1)} = 1.$ By \eqref{eq:fda} we have $H_{(1,1)}\ge 4,$ and by \eqref{eq:fbound} we have 
\begin{align}
    |U|&\ge \binom{m-1}2H_{(1,1)}-(m-2)(m-3)H_{(2,1)}+ \binom{m-3}2H_{(2,2)}\notag\\ 
    &\ge 4\binom{m-1}2 - (m-2)(m-3) = (m+1)(m-2) \ge m(m-2),
\end{align}
contradicting $|U| < m(m-2).$
\item $H_{(2,1)}\ge 2.$ We have 
\begin{align}
    |U|&\ge (m-2)(m-3)H_{(2,1)}-3\binom{m-3}2H_{(2,2)}-(m-3)H_{(4)}\notag\\
    &= (m-2)(m-3)H_{(2,1)} - (m-3)H_{(4)} \\
    &\ge (m-2)(m-3)H_{(2,1)} - (m-3)H_{(2,1)}\ge 2(m-3)^2\ge m(m-2),\notag
\end{align}
which contradicts the assumption $|U| < m(m-2),$ where we used \eqref{eq:dbound} and \eqref{eq:dab}.
\end{casesp}
\item $H_{(2,2)}=1.$
\begin{casesp}
\item $H_{(2,1)}\le 2$. By \eqref{eq:da} we have $H_{(2,1)} = 2$.
\begin{casesp}
\item $H_{(1,1)}\le 4$. By \eqref{eq:fda} we have $H_{(1,1)} = 4$. By \eqref{eq:fab} we have $H_{(4)}\ge 1$ but by \eqref{eq:dab} we have $H_{(4)}\le 1$ so $H_{(4)} = 1.$ Let 
\begin{align}h &:= 2H_{(2)} + (m-3)H_{(1,1)} - 3H_{(3)} - (3m-13)H_{(2,1)}+3H_{(4)}+(2m-11)H_{(2,2)}\notag\\
    &=2H_{(2)} + 4(m-3)-3H_{(3)} - 2(3m-13) + 3 + 2m-11\\
    &=2H_{(2)}-3H_{(3)}+6,\notag
\end{align} 
note that $H_{(1)}\ge h$ by \eqref{eq:gthree}. Then by \eqref{eq:onebound} we have 
\begin{align}
    |U| &= mH_{(1)} - \pa{(m-1)H_{(2)} + \binom{m-1}2H_{(1,1)}} \notag\\
    &+ \pa{(m-2)H_{(3)} + (m-2)(m-3)H_{(2,1)}} - \pa{(m-3)H_{(4)} + \binom{m-3}2H_{(2,2)}}\notag\\
    &= m(H_{(1)}-h) + (m+1)H_{(2)} -(2m+2)H_{(3)} - 4\binom{m-1}2 + 2(m-2)(m-3) \\
    &- \pa{(m-3)+ \binom{m-3}2} + 6m\notag\\
    &=m(H_{(1)}-h)+(m+1)(H_{(2)}-(2H_{(3)}+m-5))+\frac{m(m+1)}2.\notag
\end{align}
Note that $H_{(2)}\ge 2H_{(3)}+m-5$ by \eqref{eq:evmid}. Since $\frac{m(m+1)}2$ is either divisible by $m$ or $m+1$, we are done.
\item $H_{(1,1)}\ge 5.$ By 
\eqref{eq:fbound} we have 
\begin{align}
    |U|&\ge  \binom{m-1}2H_{(1,1)}-(m-2)(m-3)H_{(2,1)}+ \binom{m-3}2H_{(2,2)} \\
    &\ge 5 \binom{m-1}2 - 2(m-2)(m-3)+\binom{m-3}2=m^2-m-1\ge m(m-2),\notag
\end{align}
which contradicts $|U| < m(m-2).$
\end{casesp}
\item $H_{(2,1)}\ge 3.$ By \eqref{eq:fbound} and \eqref{eq:fda} we have 
\begin{align}
    |U| &\ge  \binom{m-1}2H_{(1,1)}-(m-2)(m-3)H_{(2,1)}+ \binom{m-3}2H_{(2,2)}  \notag\\
    &\ge \binom{m-1}2(4H_{(2,1)}-4) - (m-2)(m-3)H_{(2,1)} + \binom{m-3}2\notag\\
    &=H_{(2,1)}(m+1)(m-2)-2(m-1)(m-2)+\frac{(m-3)(m-4)}2\\
    &\ge(m+5)(m-2)+\frac{(m-3)(m-4)}2\ge m(m-2),\notag
\end{align}
which contradicts $|U| < m(m-2).$
\end{casesp}
\item $H_{(2,2)} = 2.$ By 
\eqref{eq:dbound} we have 
\begin{align}
    |U|&\ge (m-2)(m-3)H_{(2,1)}-3\binom{m-3}2H_{(2,2)}-(m-3)H_{(4)} \notag\\
    &= (m-2)(m-3)H_{(2,1)} - 6\binom{m-3}2 - (m-3)H_{(4)} \label{eq:dbcal}\\
    &= (m+4)(m-2) + (m-3)((H_{(2,1)}-4)(m-2)-H_{(4)}).\notag
\end{align}
We claim that $(H_{(2,1)}-4)(m-2)-H_{(4)}\ge -2$. Indeed, by \eqref{eq:da} we have $H_{(2,1)}\ge2 H_{(2,2)}= 4$ so the claim is true if $H_{(4)}\le 2.$ If $H_{(4)} > 2$ then by \eqref{eq:dab} we have \begin{align}&(H_{(2,1)}-4)(m-2)-H_{(4)}\ge (H_{(4)}-2)(m-2)-H_{(4)} \\
&= H_{(4)}(m-3)-2(m-2)\ge 2(m-3)-2(m-2)=-2.\notag\end{align}
Thus, by \eqref{eq:dbcal} we have 
\begin{equation}
    |U|\ge (m+4)(m-2)  -2(m-3)\ge m(m-2),
\end{equation}
contradicting $|U| < m(m-2).$ \qedhere
\end{casesp}
    \end{proof}
\end{lm}
\begin{proof}[Proof of Theorem \ref{thm:symsetsmallsize}]
        By \Cref{lm:symmverysmall} it suffices to show that if $S\sbs S_n$ is a symmetric set without monotone elements and $p = |S|$ then one of (1), (2), (3) holds. Note that $A(S) = (S, (A_1, \dots, A_{n-1}))$ is a complete and reduced harmonic set system. Thus, by Lemma \ref{lm:split} (with $m = n-1$) we can partition $S = S^{(1)}\sqcup S^{(2)}$ such that the induced system by $S^{(1)}$ satisfies the hypothesis of Lemma \ref{lm:calculations} and the complement of the induced system by $S^{(2)}$ satisfies the hypothesis of Lemma \ref{lm:calculations}. Thus, each $|S^{(i)}|$ is a nonnegative integer linear combination of $m, m+1, \binom m2, \frac{(m+1)(m-2)}2,$ and thus $|S|$ is a nonnegative integer linear combination of $m, m+1, \binom m2, \frac{(m+1)(m-2)}2.$ Since $m = n-1$, these are $n, n-1, \binom{n-1}2, \frac{n(n-3)}2.$

        Thus, it suffices to show that if $p$ is a nonnegative integer linear combination of $n, n-1, \binom{n-1}2, \frac{n(n-3)}2$ then one of (1), (2), (3) holds.
        Suppose 
        \begin{equation}
            p = an +b(n-1) + c \binom{n-1}2 + d\frac{n(n-3)}2.
        \end{equation}
        Note that $\binom{n-1}2 \ge \frac{n(n-3)}2 \ge \frac{(n-1)(n-3)}2$ so $c+d\le 1$, and $(n-1)(n-3)> an + b(n-1) \ge (a+b)(n-1)$ so $a +b < n-3$. If $c = d = 0$ then $p = (a+b)n - b$ so (1) holds. 

        Now consider the case $c = 1, d = 0$. If $n$ is even then 
        \begin{equation}
            p = an + b(n-1) + \binom{n-1}2 = an + \pa{b+\frac{n-2}2}(n-1),
        \end{equation}
        so the claim is reduced to the $c = d = 0$ case. If $n$ is odd and $b\ge 1$ then we have 
        \begin{equation}
            p = an + b(n-1) + \binom{n-1}2 = \pa{a+\frac{n-1}2}n + (b-1)(n-1),
        \end{equation}
        and the claim is reduced to the $c = d = 0$ case. If $b = 0$ then (3) holds. 

        Now consider the case $c = 0, d = 1$. If $n$ is odd then 
        \begin{equation}
            p = an + b(n-1) + \frac{n(n-3)}2 = \pa{a+\frac{n-3}2}n + b(n-1),
        \end{equation}
        so the claim is reduced to the $c = d = 0$ case. If $n$ is even and $a\ge 1$ then we have 
        \begin{equation}
            p = an + b(n-1) + \frac{n(n-3)}2 = (a-1)n + (b+\frac{n}2)(n-1),
        \end{equation}
        and the claim is reduced to the $c = d = 0$ case. If $a = 0$ then (2) holds.
\end{proof}

\section{Minimal symmetric and symmetrically avoided sets}\label{sec:minsymmset}
In this section, we prove Theorem \ref{thm:symmk} and Theorem \ref{thm:smallsaset} which classify symmetric sets and symmetrically avoided sets of size at most $k-1$ in $S_k.$ We start by proving some claims about harmonic set systems.
\begin{prop}\label{prop:setequal}
    If $H = (U, (A_1, \dots, A_m))$ is a harmonic set system and there exists $i< j$ such that $A_i = A_j$ then $A_1 = \dots = A_m$ or $(m,i,j)=(3,1,3)$.
    \begin{proof}
        Note that if $m = 1$ then the conclusion trivially holds, and if $m=2$ then we must have $i = 1, j = 2$ so the conclusion also holds. Thus, we can assume $m\ge 3.$

        Note that by harmonicity, all the sets have the same size $k$ for some $k$ (by taking $I = \{u\}$ and $J = \{v\}$ we have $|A_u| = |A_v|$ for any $u, v$.) 

        Note that $|A_i\cap A_j| = k$. We first consider the case $|i-j| = 1$. By harmonicity we have $|A_u\cap A_{u+1}| = k$ for all $u$. Thus, we have $|A_u| = |A_{u+1}| = |A_u\cap A_{u+1}|$ for all $u$, which implies $A_u = A_{u+1}$ for all $u$. Thus, $A_1 = \dots = A_m$.
        
        In the case $|i-j| > 1$, a similar argument to the previous case show $A_u = A_v$ for all $|u-v| > 2$. If $m > 3$ then we have $A_u = A_1$ for $u\ge 3$, and $A_2 = A_4 = A_1$, so all the sets are equal by the previous case. If $m = 3$ then since $|i - j| > 1$ we must have $i = 1, j = 3$.
    \end{proof}
\end{prop}
\begin{prop}\label{prop:harsizeone}
    Let $H = (U, \pa{A_1, \dots, A_m})$ be a harmonic set system such that $|A_i| = 1$ for each $i$. Then one of the following is true.
    \begin{enumerate}[label=\textup(\arabic*\textup)]
        \item $A_1 = \dots = A_m,$
        \item $H$ is isomorphic to the set system $\pa{[m], \pa{\{1\}, \{2\}, \dots, \{m\}}},$ or
        \item $m = 3$ and $H$ is isomorphic to  $([2],\pa{\{1\}, \{2\}, \{1\}}).$
    \end{enumerate}
    \begin{proof}
    For any $i, j$ we either have $A_i = A_j$ or $A_i\cap A_j  = \varnothing$. We have two cases 
    \begin{itemize}
        \item $A_i  = A_j$ for some $i < j$. By Proposition \ref{prop:setequal}, (1) holds unless $m = 3$ and $i = 1, j = 3$. In the case $m = 3$ and $i = 1, j = 3$, if $A_1 = A_2$ or $A_2 = A_3$ then by Proposition \ref{prop:setequal}, (1) holds. Otherwise, we have $A_1 = A_3 = \{x\}$ for some $x$ and $A_2\cap A_1  =A_3\cap A_2  = \varnothing$ so $A_2 = \{y\}$ for some $y\ne x$, so (3) holds.
        \item $A_i\cap A_j = \varnothing$ for all $i < j$. Then we have $A_i = \{a_i\}$ with $a_1, \dots, a_m$ are pairwise distinct, so (2) holds.\qedhere
    \end{itemize}
    \end{proof}
\end{prop}
\begin{lm}\label{lm:4to2}
    There does not exist a harmonic set system $H = (U, (A_1, A_2, A_3, A_4))$ with $U\sbs [4]$ such that $A_4 = [2]$ and $A_i\cap A_4 = \{i\}$ for $i\in \{1, 2\}$ and $A_2 = \{2,3\}.$ 
    \begin{proof}
        Note that we have $|A_3\cap (A_2\cup A_4)| = |A_3|+|A_2\cup A_4| - |A_3\cup A_2\cup A_4|\ge 2 + 3 -4 = 1$ so $A_3\cap (A_2\cup A_4)\ne\varnothing$, so $A_3\cap A_2$ or $A_3\cap A_4$ is nonempty, by harmonicity we have $A_i\cap A_{i+1}$ is nonempty for $1\le i\le 3$. Thus, $A_1\cap A_2$ is nonempty so the other element of $A_1$ must be $3$, that is, $A_1 = \{1, 3\}$. Now if $A_3\subseteq [3]$ then note that it must be equal to one of $A_1, A_2, A_4$ since $A_1, A_2, A_4$ are all possible subsets of size $2$ of $[3]$, which by Proposition \ref{prop:setequal} implies $A_3 = A_4$, a contradiction. So we may assume $4\in A_3$. Then let $a_3$ be the other element of $A_3$, then $[3]\sm \{a_3\}$ must be one of $A_1, A_2, A_4$ since it is a subset of size $2$ of $[3]$. Since $[3]\sm \{a_3\}\cap A_3 = \varnothing$ we have $A_3 \cap A_j = \varnothing$ for some $j\in\{1, 2, 4\}$, but we have $A_2\cap A_3, A_3\cap A_4$ are nonempty so $j = 1$. But then $|A_1\cap A_3| = 0\ne 1 = |A_2\cap A_4|$, which contradicts harmonicity.  
    \end{proof}
\end{lm}
\begin{lm}\label{lm:523}
    Suppose $H = (U, (A_1, A_2, A_3, A_4, A_5))$ is a harmonic set system with $U\sbs [5]$ such that $A_5 = \{1,2\}, A_3 = \{1,3\}, A_1\cap A_5 = A_3\cap A_5 =  \{1\}$ and $ A_2\cap A_5 = \{2\}.$ Then $H$ is isomorphic to $([5], (\{1,4\}, \{2,5\}, \{1,3\}, \{4,5\}, \{1,2\})).$
    \begin{proof}
        Since $|A_3\cap A_5| = 1$ we have $|A_1\cap A_3| = 1$ so the other element of $A_1$ cannot be $3$, and it cannot be $2$ since $2\notin B_1$, so without loss of generality assume it's $4$, that is, $A_1 = \{1, 4\}$. By harmonicity we have $|A_2\cap A_3| = |A_2\cap A_1|$. If $A_2\cap A_3| = |A_2\cap A_1|\ge 1$ then we have $4\in A_2$ (since $A_2\cap A_1\ne\varnothing$) and $3\in A_2$ (since $A_2\cap A_3\ne\varnothing$), so $|A_2|\ge 3$, which is a contradiction. Thus, we must have $|A_2\cap A_3| = |A_2\cap A_1| = 0$, so $1, 3, 4\notin A_2$. Thus, $A_2 = \{2, 5\}$. By harmonicity we have $|A_3\cap A_4| = |A_4\cap A_5| = 0$, so $1, 2, 3\notin A_4$, so $A_4 = \{4,5\}$. 
    \end{proof}
\end{lm}
\begin{lm}\label{lm:star}
    For $m\ge 4$, there does not exist a harmonic set system $H = (U, (A_1, \dots, $ $A_m))$ with $U\sbs [m]$ such that $A_m = [2]$ and $A_i\cap A_m = \{1\}$ for $1\le i\le m-2.$
    \begin{proof}
        By Proposition \ref{prop:setequal} we have $A_i\ne A_j$ for all $i\ne j$ (otherwise $[2] = A_m  = A_1 = A_1\cap A_m = \{1\}.$)
        For each $1\le i\le m-2$ let $a_i$ denotes the other element of $A_i$, that is, $A_i = \{1, a_i\}$ for $1\le i\le m-2$; and denote $a_{m} = 2$ so that $A_m = \{1, a_m\}$. Since $A_j\ne A_j$ for $i\ne j$, we have  $a_i\ne a_j$ for $i, j\in [m]\sm\{m-1\}, i\ne j$. This implies $|A_1\cap A_2| = 1$, so by harmonicity we have $|A_i\cap A_j| = 1$ when $|i-j| = 1$. But we have $|A_1\cap A_m| = 1$ which implies $|A_i\cap A_j| = 1$ when $|i-j|>1$. Thus, $|A_i\cap A_j| = 1$ for $i\ne j$. If $A_{m-1}$ does not contain $1$ then since $|A_{m-1}\cap A_i| = 1$ for all $i\in [m]\sm \{m-1\}$ we must have $a_i\in A_{m-1}$ for all $i\in [m]\sm \{m-1\}$, so $|A_{m-1}|\ge m-1\ge 3$, which is a contradiction. Thus, we have $1\in A_{m-1}$, and define $a_{m-1}$ to be the other element in $A_{m-1}.$ Since $A_j\ne A_j$ for $i\ne j$, we have $a_i \ne a_j$ for some $1\le i < j \le m$ so $1, a_1, \dots, a_m$ are $m+1$ distinct elements in $U$, which is a contradiction.
    \end{proof}
\end{lm}
\begin{lm}\label{lm:bcharmonic}
    Suppose $H = (U, (A_1, \dots, A_m))$ is a harmonic set system. For $1\le i\le m-2$ define $B_i := A_i\cap A_m$ and $C_i := A_i\sm A_m.$ Then $B := (A_m, (B_1, \dots, B_{m-2}))$ and $C := (A_m, (C_1, \dots, C_{m-2}))$ are harmonic set systems.
    \begin{proof}
        For any $I, J\subseteq [m-2]$ with the same run decomposition, note that $I \cup \{m\}$ and $J\cup \{m\}$ have same run decomposition, so we have 
        \begin{equation}
            \abs*{\bigcap_{i\in I}B_i} = \abs*{\bigcap_{i\in I \cup \{m\}}A_i} = \abs*{\bigcap_{j\in J \cup \{m\}}A_j} = \abs*{\bigcap_{j\in J}B_j},
        \end{equation}
        and 
        \begin{align}
            \abs*{\bigcap_{i\in I}C_i} &= \abs*{\pa{\bigcap_{i\in I }A_i}\sm A_m} = \abs*{\bigcap_{i\in I}A_i} - \abs*{\pa{\bigcap_{i\in I }A_i}\cap A_m} = \abs*{\bigcap_{i\in I}A_i} - \abs*{\pa{\bigcap_{i\in I \cup \{m\}}A_i}}\notag\\
            &= \abs*{\bigcap_{j\in J}A_j} - \abs*{\pa{\bigcap_{j\in J \cup\{m\}}A_j}}   = \abs*{\pa{\bigcap_{j\in J }A_j}\sm A_m} = \abs*{\bigcap_{j\in J}C_j}.\qedhere
        \end{align}
    \end{proof}
\end{lm}
\begin{lm}\label{lm:harmonicm}
    Suppose $H = (U, \pa{A_1, \dots, A_m})$ is a harmonic set system such that $|U|\le m$. Then one of the following is true.
    \begin{enumerate}[label=\textup(\arabic*\textup)]
        \item $A_1 = \dots = A_m,$
        \item $H$ is isomorphic to the set system $\pa{[m], \pa{\{1\}, \{2\}, \dots, \{m\}}}$ or its complement,
        \item $m = 3$ and $H$ is isomorphic to  $([2],\pa{\{1\}, \{2\}, \{1\}})$ or $([3],\pa{\{1\}, \{2\}, \{1\}})$ or \\$([3],(\{1,2\}, \{2,3\}, \{1,2\}))$, or
        \item $m = 5$ and $H$ is isomorphic to the set system $([5],\pa{\{1,4\}, \{2,5\}, \{1,3\}, \{4,5\}, \{1,2\}})$ or its complement.
    \end{enumerate}
    \begin{proof}
        We use induction on $m.$ Since the set system is harmonic, there is some $k$ such that $|A_i| = k$ for all $i$. Note that all the four properties above and harmonicity are preserved under taking complement (see Corollary \ref{cor:comharset}). Thus, we may without loss of generality assume $k\le \frac m2$, otherwise we can replace $H$ with its complement.

        Note that if $k = 0$ then $A_1 = \dots = A_m =  \varnothing$, so (1) holds. If $k = 1$ then one of (1), (2) or (3) holds by Proposition \ref{prop:harsizeone}. Thus we may assume $k\ge 2$. 
        
        Since $k\le \frac m2$ we have $m\ge 4$. For $1\le i\le m-2$ define $B_i := A_i\cap A_m$ and $C_i := A_m\sm A_i,$ and define $B := (A_m, (B_1, \dots, B_{m-2}))$ and $C := (U\sm A_m, (C_1, \dots, C_{m-2}))$, which are harmonic set systems by \Cref{lm:bcharmonic}.  

        Note that we have $|A_m| = k\le \frac m2\le m-2$ and $\abs*{U\sm A_m} \le m - k\le  m-2$, so $B$ and $C$ are harmonic set systems satisfying the inductive assumption.
         
        Now we consider the case $k = 2$. Without loss of generality assume $U \subseteq [m]$ and $A_ m = \{1,2\}$. Note that if $m = 4$ then $m-2 = 2$ and so $(B_1, \dots, B_{m-2})$ can only satisfy either (1) or (2). If $m > 4$ then $m-2 > 2$, so  $|A_m| <  m-2$, so $B$ can only satisfy (1) or (3) since (2) and (4) requires the size of the universe set to be exactly equal to the number of sets. Also note that $B$ can only satisfy (3) when $m-2 = 3$ which is when $m = 5$. Thus, we have three cases:
        \begin{itemize}
            \item $B$ satisfies (1). Note that if $|B_i| = 2$ then $A_i = A_m$ and by Proposition \ref{prop:setequal} (1) holds. If $|B_i| = 1$ then without loss of generality assume $B_i = \{1\}$ for $1\le i\le m-2$, and we have a contradiction by \Cref{lm:star}. Thus, $|B_i| = 0$. Then $|A_1\cap A_m| = 0$  and by harmonicity $|A_i\cap A_j| = 0$ for $|i-j|>1$. Thus, $A_1, A_3, \dots, A_{2\ceil*{\frac m2}-1}$ are $\ceil*{\frac m2}$ pairwise disjoint sets of size $2$ each, so we have their union has size $2\ceil*{\frac m2}$, so $m\ge 2\ceil*{\frac m2}$, so $m$ is even.
                Since $B_i = \varnothing$ we have $C_i = A_i$ for all $1\le i\le m-2$.
                Since $m$ is even, $m-2$ is even, so $C$ cannot satisfy properties (3) or (4), and thus must satisfy (1) or (2). But it cannot satisfy (2) since (2) requires the sets to have size $1$ or $m-2-1 = m-3$ which are odd numbers, whereas $|C_i| = 2$. Thus, it satisfies (1) which means $C_1 = C_2$ and so $A_1 = A_2$ which by Proposition \ref{prop:setequal} implies (1) holds for the original set system $H$.
            \item $m = 4$ and $B$ satisfies (2). Without loss of generality, assume $B_1 = \{1\}, B_2 = \{2\}$, $A_2 = \{2, 3\}$. This is a contradiction by \Cref{lm:4to2}.
            \item $m= 5$ and $B$ satisfies (3). Note that $B$ must be isomorphic to $([2], (\{1\}, \{2\}, \{1\}))$ since $|A_m| = 2$. Without loss of generality assume $B_1= B_3 =\{1\}, B_2 = \{2\}, A_3 = \{1, 3\}$. Then (4) holds by \Cref{lm:523}.
        \end{itemize}
        Finally, we consider the case $2<k\le \frac m2$. Note that $m\ge 6$. Then note that $|A_m|< m-2$, and note that  $\abs*{U\sm A_m}\le m-k < m-2$. Thus, both $B$ and $C$ cannot satisfy (2) or (4) since the size of the universe is strictly less than $m-2$, the number of set. Since $m\ge6$ we have $m-2\ge 4$ so both $B$ and $C$ cannot satisfy (3). Thus, both $B$ and $C$ must satisfy (1). This means $B_1 = B_2$ and $C_1 = C_2$ so $A_1 = B_1\cup C_1 = B_2\cup C_2 = A_2$, so by Proposition \ref{prop:setequal}, (1) holds for $H$.
    \end{proof}
\end{lm}
\begin{proof}[Proof of Theorem \ref{thm:symmk}]
    Note that $(3)$ clearly implies $ (1)$. We will show $(1)\implies (2)$ and $(2)\implies (3)$.

    We first show  $(1) \implies (2)$. Let $S' = S\sm \{\iota_n, \delta_n\}$. Since any subset of $\{\iota_n, \delta_n\}$ is symmetric, we have $S'$ is symmetric by Proposition \ref{prop:permcupsm}. Note that $A(S') = (S', (A_1, \dots, A_{k-1}))$ is harmonic, reduced and complete (Proposition \ref{prop:symm2harset} and Proposition \ref{prop:conid}). Since $|S'|\le |S|\le k-1$, by Lemma \ref{lm:harmonicm} we have $A(S')$ satisfies one of the four properties (1)---(4) of \Cref{lm:harmonicm}. 
    \begin{casesp}
        \item $A(S')$ satisfies (1) of Lemma \ref{lm:harmonicm}. Since $A(S')$ is reduced we must have $A_i = \varnothing$ for all $i$, and $A(S')$ is complete so $S' = \bigcup\limits_{i = 1}^{n-1}A_i = \varnothing$. So $S \subseteq \{\iota_n, \delta_n\}$.
        \item $A(S')$ satisfies (2) of Lemma \ref{lm:harmonicm}. By taking complement, without loss of generality assume $A(S')$ is isomorphic to $\pa{\{1\}, \dots, \{n-1\}}$. Then each $A_i = \{\pi_i\}$ for some $\pi_i\in S'$ such that all the $\pi_i$ are distinct. Since $S\supseteq S'$ has at most $k-1$ elements, we have $S = S' = \{\pi_1, \dots, \pi_{n-1}\}$, and note that $\Des(\pi_i) = \{i\}$.
        \item $A(S')$ satisfies (3) of Lemma \ref{lm:harmonicm}. This means $n-1 = 3$ so $n = 4$. Note that $A(S')$ is isomorphic to $\pa{[2], \pa{\{1\}, \{2\}, \{1\}}}$ since the other two set system are either not reduced or not complete. Then $A_1 = A_3 = \{\pi_1\}$ and $A_2 = \{\pi_2\}$ for some distinct $\pi_1, \pi_2\in S$. Thus we have $S' = \{\pi_1, \pi_2\}$, with $\Des(\pi_1) = \{1, 3\}$ and $\Des(\pi_2) = \{2\}$. Since $|S| \le 3$ and since $S\sm S'\subseteq \{\iota_n, \delta_n\}$, we have $S = \{\pi_1, \pi_2\}$ or $S = \{\pi_1, \pi_2, \iota_n\}$ or $S = \{\pi_1, \pi_2, \delta_n\}$.
        \item $A(S')$ satisfies (4) of Lemma \ref{lm:harmonicm}. This means $n- 1 = 5$ so $n = 6$. By taking complement, without loss of generality assume $A(S')$ is isomorphic to $([5],(\{1,4\}, \{2,5\},$ $\{1,3\}, \{4,5\}, \{1,2\})).$ Then $A_1 =\{\pi_1, \pi_4\},$  $A_2 = \{\pi_2, \pi_5\}$, $A_3 = \{\pi_1, \pi_3\}$, $A_4 = \{\pi_1, \pi_5\}$, $A_5 = \{\pi_1, \pi_2\}$, for $\pi_1, \dots, \pi_5$ are $5$ distinct elements in $S'$. Since $|S'|\le|S|\le 5$ we have $S = S' = \{\pi_1, \dots, \pi_5\}$, with $\Des(\pi_1) = \{1, 3, 5\}$, $\Des(\pi_1) = \{2,4\}$, $\Des(\pi_3) = \{4\},$ $\Des(\pi_4) = \{1, 4\},$ $\Des(\pi_4) = \{2, 4\}.$ 
    \end{casesp}

    Now we show $(2)\implies (3)$. Note that $\iota_n, \delta_n$ has corresponding quasisymmetric generating function $s_{(n)}$, $s_{(1^n)}$ respectively, so if $S$ satisfies (a) then $S$ is Schur-positive. If $S$ satisfies (b), then by taking complement, without loss of generality assume $\Des(\pi_i) = \{i\}$ for all $i$. 
    Then \begin{equation}
        Q(S) = \sum_{w\in S}F_{n, \Des(w)} = \sum_{i = 1}^{n-1}F_{n, \{i\}}
    \end{equation}
    We claim that the above sum is equal to $s_{(n-1, 1)}$. Indeed, note that $(n-1,1)$ has $n-1$ SYT, which can be obtained by choosing one of $2, 3,\dots, n$ into the box on the second row and putting the remaining numbers in ascending order in the first row. Then the descent set is $\{a-1\}$ where $a$ is the number in the second row. Thus, by Theorem \ref{thm:slambda} \begin{equation}
        s_{(n-1, 1)}= \sum_{a = 2}^{n}F_{n, \{a-1\}}=\sum_{i = 1}^{n-1}F_{n, \{i\}} = Q(S)
    \end{equation}
    as claimed. 
    
    If $S$ satisfies (c), then the quasisymmetric generating function is $F_{4, \{2\}} + F_{4, \{1, 3\}} + F'$
    where $F' = 0$ if $S = \{\pi_1, \pi_2\}$ and $F' = s_{(1^n)}$ if $S = \{\pi_1, \pi_2, \delta_n\}$ and $F' = s_{(n)}$ if $S = \{\pi_1, \pi_2, \iota_{n}\}$. Since $F'$ is Schur-positive, it suffices to show $F_{4, \{2\}} + F_{4, \{1, 3\}}$ is Schur-positive. But note that $(2,2)$ has two SYT with descent sets $\{1, 3\}$ and $\{2\}$ respectively, so $F_{4, \{2\}} + F_{4, \{1, 3\}} = s_{(2,2)}$ by Theorem \ref{thm:slambda}.

    If $S$ satisfies (d), then by taking complement, without loss of generality assume we are in the first case. Then 
    \begin{equation}
        Q(S) = F_{6, \{1,3,5\}}+ F_{6, \{2,5\}}+F_{6, \{3\}}+F_{6, \{1,4\}}+F_{6, \{2,4\}}
    \end{equation}
    One can check that the partition $(3,3)$ has exactly $5$ SYT with exactly the above descent sets, and thus 
    \begin{equation} 
        s_{(3,3)} =  F_{6, \{1,3,5\}}+ F_{6, \{2,5\}}+F_{6, \{3\}}+F_{6, \{1,4\}}+F_{6, \{2,4\}} = Q(S)
    \end{equation} by Theorem \ref{thm:slambda}.
\end{proof} 
    
We now turn our attention to symmetrically avoided sets of sizes at most $k-1.$
\begin{defn}[{\cite[Section 4]{bloomsagan}}]\label{def:partialsh}
    For any $a\le k,$ define the following \textit{partial shuffle} in $S_k$:
    \begin{equation}
        (1, 2, \dots,\what{a}, \dots, k) \cshuffle (a) = (1, 2, \dots, \what{a},\dots, k)\shuffle (a) \sm \{\iota_k\}
    \end{equation}
    where  $\what{a}$ denotes deletion of the element $a$ and $\shuffle$ denotes the standard shuffle. For example, 
    \begin{equation}
        (1,2, \what{3}, 4, 5, 6)\cshuffle (3) = \{312456, 132456, 124356, 124536, 124563\}
    \end{equation}
\end{defn}

It was proven in \cite[Theorem 4.1]{bloomsagan} that partial shuffles are Schur-positively avoided, and in \cite[Proposition 2.6]{HZP} that subsets of $\{\iota_k, \delta_k\}$ are Schur-positively avoided. By \Cref{cor:crpat}, complement of partial shuffles are Schur-positively avoided as well. Theorem \ref{thm:smallsaset} claims that in fact, partial shuffles, their complement, and subsets of $\{\iota_k, \delta_k\}$ are the only symmetrically avoided sets of size at most $k-1$. 
The theorem will be proved through a series of lemmas. We will first prove some lemmas about permutations with only one descent.
\begin{lm}\label{lm:countconone}
    Fix $1 \le i < j < k$. Let $\sigma\in S_k$ with $\Des(\sigma) = \{i\}$ or $\Des(\sigma) = \{j\}.$ Then there are $k$ permutations $\pi\in S_{k+1}$ containing $\sigma$ such that $\Des(\pi) = \{i, j+1\}.$ 
    \begin{proof}
        
        We consider the case $\Des(\sigma) = \{i\},$ the proof for the case $\Des(\sigma) = \{j\}$ is analogous.
        Note that if $\pi$ contains $\sigma$, 
        we can obtain a plot of it by adding one more point to the standard plot of $\sigma$. The vertical and horizontal lines passing through the points in the plot splits the plane into $(k+1)^2$ regions (see Figure \ref{fig:addpoint}.)
        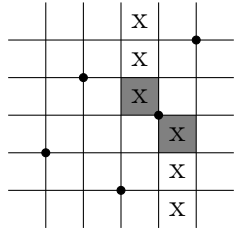
\begin{figure}[h!]
            \begin{tikzpicture}[scale=0.5]
                \fill[gray] (3,3) rectangle (4,4);
        \fill[gray] (4,2) rectangle (5,3);
                \filldraw (1,2) circle (3pt);
        \filldraw (2,4) circle (3pt);
        \filldraw (3,1) circle (3pt);
        \filldraw (4,3) circle (3pt);
        \filldraw (5,5) circle (3pt);
        \foreach \i in {1,2,3,4,5}
        {\draw[thin] (\i, 0)--(\i, 6);
        \draw[thin] (0, \i)--(6,\i);
        
        }
        
        \node at (3.5, 3.5){x};
        \node at (3.5, 4.5){x};
        \node at (3.5, 5.5){x};
        \node at (4.5, 2.5){x};
        \node at (4.5, 1.5){x};
        \node at (4.5, 0.5){x};
            \end{tikzpicture}
            \caption{Here, $k = 5$, $i = 2, j = 3$. Adding a point to one of the regions indicated with an x creates a permutation with descents at $2$ and $4$. The two grayed regions give the same permutations.}\label{fig:addpoint}
        \end{figure}
        Since $\sigma$ has only one descent at $i$, to create a permutation with descent at $i$ and $j+1$, this new point must create the new descent. That is, it can be added between $x = j-1$ and $x = j$ line (so that it becomes the new $j$-th point from the left), and must be higher than the 
        $y = \sigma(j)$ line to create a descent; or it can be added between $x = j$ and $x = j+1$ line (so that it becomes the new $j+1$-th point from the left), and must be lower than the $y = \sigma(j)$ line to create a descent. Thus, there are in total $k+1$ possible regions to add this new point. However, note that adding a point to the $(j-1, j)\times (\sigma(j), \sigma(j)+1)$ gives the same permutation as adding a point to $(j, j+1)\times (\sigma(j)-1, \sigma(j))$ (see Figure \ref{fig:addpoint}). Thus there are $k$ total distinct permutations.
    \end{proof}
\end{lm}
\begin{defn}
    Let $i, t, k$ be integers with $1\le i,t \le k.$
    Define $Q_k(i, t)$ to be the permutation $1,2, \dots, i-1, t, i, \dots, \what t, \dots, k$ if $t > i$ and $1,2, \dots, \what t,\dots, i+1,t, i+2, \dots, k$ if $t\le i.$ We may drop the subscript $k$ when it's clear from context.
\end{defn}
\begin{rmk}\label{rmk:caseworkg}
    For $\sigma\in S_k$, we have that $\sigma$ contains $\iota_{k-1}$ if and only if there exists $i,t$ such that $\sigma = Q(i,t).$ In this case, $\Des(\sigma) = \{i\}.$  
    Note that $Q_k(i,i) = Q_k(i,i+1)$ for all $i$, but otherwise all $Q_k(i,t)$ are distinct. See Figure \ref{fig:addpoint2}.
\end{rmk}

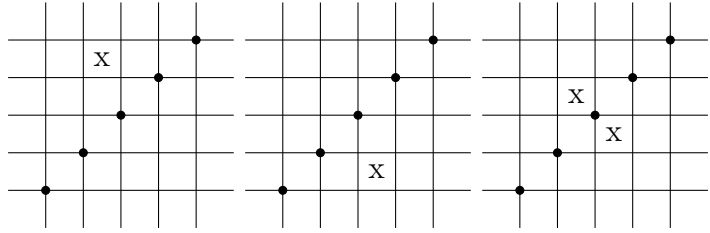
\begin{figure}[h!]
    \begin{tikzpicture}[scale=0.5]
\foreach \i in {1,2,3,4,5}
{\draw[thin] (\i, 0)--(\i, 6);
\draw[thin] (0, \i)--(6,\i);
\filldraw (\i,\i) circle (3pt);
}
\node at (2.5, 4.5){x};
\end{tikzpicture}
\begin{tikzpicture}[scale=0.5]
    \foreach \i in {1,2,3,4,5}
    {\draw[thin] (\i, 0)--(\i, 6);
    \draw[thin] (0, \i)--(6,\i);
    \filldraw (\i,\i) circle (3pt);
    }
    \node at (3.5, 1.5){x};
\end{tikzpicture}
\begin{tikzpicture}[scale=0.5]
    \foreach \i in {1,2,3,4,5}
    {\draw[thin] (\i, 0)--(\i, 6);
    \draw[thin] (0, \i)--(6,\i);
    \filldraw (\i,\i) circle (3pt);
    }
    \node at (3.5, 2.5){x};
    \node at (2.5, 3.5){x};
\end{tikzpicture}
    \caption{Here, $k = 6, i = 3$. Left: Add a point in the $(2, 3)\times (4,5)$ region to obtain $Q(3,5) = 125346$. Middle: Add a point in the $(3, 4)\times (1,2)$ region to obtain $Q(3,2) = 134256$. Right: Add a point to $(2, 3)\times (3,4)$ or to $(3, 4)\times (2,3)$ to obtain $Q(3,3) = Q(3,4) = 124356$.}
    \label{fig:addpoint2}
\end{figure}
\begin{defn}
    For a permutation $\pi$, define $\del(\pi,i)$ to be the sequence obtained by deleting the $i$-th letter of $\pi$ (in one-line notation.)
\end{defn}
For all the remaining lemmas in this section, we will have the following setup. Let $\Pi= \{\sigma_1, \dots, \sigma_{k-1}\}\sbs S_k$ be such that $\Des(\sigma_i) = \{i\}$ for each $i$.
\begin{lm}\label{lm:countcontwo}
    For $1\le i < k$, define 
    \begin{equation}
        P_i:= \{\std(\del(\sigma_i, i)), \std(\del(\sigma_i, i+1))\},
    \end{equation}
    which may have only one element if the two deletions are the same.
    Fix $i< j$, and let $N$ be the number of permutation $\pi\in S_{k+1}$ containing both $\sigma_i$ and $\sigma_j$ with $\Des(\pi) = \{i, j+1\}.$ Then 
    \begin{enumerate}[label=\textup(\arabic*\textup)]
        \item $N= 0$ if and only if $P_i\cap P_j = \varnothing.$
        \item If $P_i\cap P_j = \{\iota_{k-1}\}$ and there is no $t$ such that $\sigma_i = Q(i,t)$ and $\sigma_j = Q(j,t)$ then $N = 1.$
        \item If $P_i\cap P_j \sps \{\iota_{k-1}\}$ and there is a $t$ such that $\sigma_i = Q(i,t)$ and $\sigma_j = Q(j,t)$ then $N\ge 2.$
    \end{enumerate}
    \begin{proof}
        We first show (1). Suppose $w \in P_i\cap P_j$. Then one can obtain a plot of $\sigma_i$ from the standard plot of $w$ by adding a point that creates a descent at $i$, i.e. adding a point $A$ in $(i-1, i)\times (w(i), k)$ or in $(i, i+1)\times (0,w(i))$. Similarly, one can obtain a plot of $\sigma_j$ from the standard plot of $w$ by adding a point $B$ that creates a descent at $j.$
        If we add both of $A$ and $B$ in such a way that they are not on the same horizontal line and $A$ is to the left $B$, then we get a permutation of length $k+1$ that contains both $\sigma_i, \sigma_j$ and has descent at $i, j+1$ (the descents at $j$ get reindexed into $j+1$ due to the adding of the $A.$)

        Conversely, suppose $\pi\in S_{k+1}$ contains both $\sigma_i, \sigma_j$ and $\Des(\pi) = \{i, j+1\}$. Let $P$ be the standard plot of $\pi$. Then one can delete one point $A$ from $P$ to obtain a plot of $\sigma_i$, which must delete the $j+1$ descent, and thus must be the $j+1$-th or $j+2$-th point.
        Similarly, one can delete one point $B$ from the standard plot of $\pi$ to obtain $\sigma_j$, which must be either the $i$ or $i+1$-th point.
        We have $A\ne B$ since $\sigma_i\ne \sigma_j$. Then $P\sm \{A, B\}$ is obtainable by starting with $P\sm \{A\}$, which is a plot of $\sigma_i$, and delete $B$, the $i$ or $i+1$-th point. Also, $P\sm \{A, B\}$ is obtainable by starting with $P\sm \{B\}$, which is a plot of $\sigma_j$, and delete $A$, the $j$ or $j+1$-th point (the deletion of $B$ from $P$ reindex $A$ into $j$ or $j+1$.) Thus, $P\sm \{A, B\}$ is a plot of a permutation in $P_i\cap P_j$.

        To show (2), note that we proved above that permutations $\pi\in S_{k+1}$ containing both $\sigma_i$ and $\sigma_j$ with $\Des(\pi) = \{i, j+1\}$ must be obtained by starting with a plot of a permutation $w\in P_i\cap P_j$ and add two points, where adding one creates a plot of $\sigma_i$---call this point $A$, while adding the other creates a plot of $\sigma_j$---call this point $B$. We must have $w = \iota_{k-1}$. Without loss of generality, assume we start with the standard plot of $\iota_{k-1}$.

        Note that the $x$-coordinate of $A$ must be in $(i-1,i+1)$ and the $x$-coordinate of $B$ must be in $(j-1,j+1).$ If $j > i+1$ then $A$ is to the left of $B$. If $j  = i+1$, we still have $A$ is to the left of $B$, since otherwise we must have $A\in (i, i+1)\times (0, i)$ and $B\in (i, i+1)\times (i, k)$, which would create a permutation with only one descent.

        Note that the $y$-coordinate of $A$ can be in the range $(t,t+1)$ if and only if $\sigma_i = Q(i,t)$. By \Cref{rmk:caseworkg} and the fact that there is no common $t$ such that $\sigma_i = Q(i,t)$ and $\sigma_j = Q(j,t)$, we deduce that the relative order of the $y$-coordinates of $A$ and $B$ is also determined. Therefore, we get the same permutation regardless of the way we added $A$ and $B$.

        To prove (3), note that one can start with the standard plot of $\iota_{k-1}$ and add two points $A$ and $B$ as above. Since $\sigma_i = Q(i,t)$ and $\sigma_j = Q(j,t)$, it is possible to change the relative ordering of the $y$-coordinates of $A$ and $B$. Thus, we can get $2$ different permutations, so $N \ge 2$.
    \end{proof}
\end{lm}
\begin{lm}\label{lm:p1p2}
    Define $P_i$ as in Lemma \ref{lm:countcontwo}. Suppose $\sigma_1 = Q(1,t)$ and $\sigma_4 = Q(4,t)$ for some $t > 2$ and there is $t'$ such that $\sigma_2 = Q(2,t')$ and $\sigma_4 = Q(4,t')$. Then either $t = t'$ or $P_1\cap P_2 = \{\iota_{k-1}\}.$
    \begin{proof}
        Since $Q(4,t) = \sigma_4 = Q(4,t'),$ by \Cref{rmk:caseworkg}, we have $t = t'$ or $\{t, t'\} = \{4,5\}.$
        If $\{t, t'\} = \{4,5\}$ one can directly check that $P_1\cap P_2 = \{\iota_{k-1}\}.$
    \end{proof}
\end{lm}
\begin{lm}\label{lm:countconf}
    There are
    $k-1$ permutations $\pi \in S_{k+1}$ containing $\sigma_1$ with $\Des(\pi) = \{1,2\}.$ 
    \begin{proof}
        Note that $\sigma_1 = Q_k(1,t)$ for some $1 < t \le k.$ Then the permutations in $S_{k+1}$ containing $\sigma$ with descent set $\{1,2\}$ are $z, t, 1,\dots, \what t , \dots,\what z, k+1$ for $k+1\ge z > t$ and $t+1, z, 1, \dots \what{z}, \dots, \what{t+1}, \dots, k+1$ for $t> z > 1.$
    \end{proof}
\end{lm}
\begin{lm}\label{lm:countaddnodes}
    There are $k+1-i$ permutations $\pi \in S_{k+1}$  containing $\sigma_i$ with $\Des(\pi) = \{i+1\}$ and there are $i+1$ permutations $\pi \in S_{k+1}$ containing $\sigma_i$ with $\Des(\pi) = \{i\}.$
    \begin{proof}
        We first count the permutations $\pi \in S_{k+1}$ containing $\sigma_i$ with $\Des(\pi) = \{i+1\}$. We want to add one more point to the standard plot of $\sigma$ in such a way that it creates the plot of a permutation with unique descent at $\{i+1\}$. To do this, one need to add a point to the left of the $x = i$ line so that the current descent at $i$ gets reindexed into $i+1$, and it should be added into a region that does not create new descent. That is, it must be added to $(a-1, a)\times (y_{a}, y_{a+1})$ for some $a\le i$, with 
        \begin{equation}
            y_a := \begin{cases}
                0, &\text{if } a = 0,\\
                \sigma_i(a),& \text{if }1\le a \le i,\\
                k+1,&\text{if }a = i+1.
            \end{cases}
        \end{equation}
        (See Figure \ref{fig:addpointbefore}.) 
        \begin{figure}[h!]
        \begin{tikzpicture}[scale=0.5]
            \foreach \i in {1,2,3,4,5,6}
            {\draw[thin] (\i, 0)--(\i, 7);
            \draw[thin] (0, \i)--(7,\i);
            }
            \filldraw (1,1) circle (3pt);
            \filldraw (2,2) circle (3pt);
            \filldraw (3,5) circle (3pt);
            \filldraw (4,3) circle (3pt);
            \filldraw (5,4) circle (3pt);
            \filldraw (6,6) circle (3pt);
            \node at (0.5, 0.5){a};
            \node at (1.5, 1.5){a};
            \node at (2.5, 2.5){a};
            \node at (2.5, 3.5){b};
            \node at (2.5, 4.5){c};
            \node at (3.5, 5.5){c};
            \node at (3.5, 6.5){d};
        \end{tikzpicture}
        \begin{tikzpicture}[scale=0.5]
            \foreach \i in {1,2,3,4,5,6}
            {\draw[thin] (\i, 0)--(\i, 7);
            \draw[thin] (0, \i)--(7,\i);
            }
            \filldraw (1,1) circle (3pt);
            \filldraw (2,2) circle (3pt);
            \filldraw (3,4) circle (3pt);
            \filldraw (4,5) circle (3pt);
            \filldraw (5,3) circle (3pt);
            \filldraw (6,6) circle (3pt);
            \node at (4.5, 0.5){a};
            \node at (4.5, 1.5){b};
            \node at (4.5, 2.5){c};
            \node at (5.5, 3.5){c};
            \node at (5.5, 4.5){d};
            \node at (5.5, 5.5){e};
            \node at (6.5, 6.5){e};
        \end{tikzpicture}
            \caption{Here, $k = 6.$ Left: adding a point in one of regions with letters creates a permutation in with descent set $\{4\}$. Regions with the same letter give the same permutation. Right: adding point in one of regions with letters creates a permutation with descent set $\{4\}.$}
            \label{fig:addpointbefore}
        \end{figure}
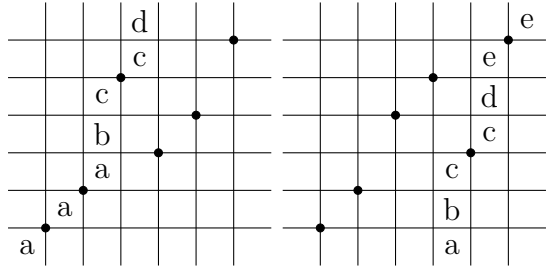
        There are a total of $k+1$ such regions, since in each of the $k+1$ row there is exactly one region. However, there are $i$ points of the original permutations of the form $(a, \sigma_i(a))$ with $1\le a\le i$. Two regions that are diagonally connected by these points will give the same permutation. Since there are $i$ points and each point merge two regions into one, the number of different permutations is $k+1-i.$ 

        We now count the permutations $\pi \in S_{k+1}$  containing $\sigma_i$ with $\Des(\pi) = \{i\}$. We need to add one point to the standard plot of $\sigma_i$ to create a permutation with unique descent at $i$. To do this, one need to add a point to the right of the $x = i$, and it should be added into a region that does not create new descent. That is, the new point must be added to $(a, a+1)\times (y_a, y_{a+1})$ for some $a\ge i$, where 
        \begin{equation}
            y_a =\begin{cases}
                0, &\text{if } a = i,\\
                \sigma_i(a), &\text{if }i < a < k,\\
                k+1, &\text{if }a = k.
            \end{cases}
        \end{equation}
        (See Figure \ref{fig:addpointbefore}.) Note that we also have $k+1$ different regions, and there are $k- i$ points of the original permutation $\sigma_i$ each of which merge two regions into one, so there are $(k+1) - (k-i) = i+1$ different permutations.
    \end{proof}
\end{lm}
\begin{lm}\label{lm:inductstep}
    Suppose that $\sigma_{i-1} = Q(i-1, t)$ for some $t > i$ and $\sigma_i = Q(i, t')$ for some $t'.$ If there exists $\pi\in S_{k+1}$ containing both $\sigma_{i-1}$ and $\sigma_i$ with $\Des(\pi) = \{i\}$, then $t' = t$ or $t' = t+1.$
    \begin{proof}
        Note that the permutations in $S_{k+1}$ containing $Q(i-1,t)$ with descent set $\{i\}$ are \begin{equation}1,\dots, i-2, z, t+1, i-1, \dots, \what z,\dots, \what{t+1}, \dots, k+1\end{equation} for $i-1\le z < t$ and \begin{equation}1, \dots, i-2, t, z, i-1, \dots,\what {t}, \dots, \what{z}, \dots, k+1\label{eq:largez}\end{equation} for $t < z \le k+1.$
        Thus, $\pi$ must be among the above permutations. In particular,
        \begin{equation}
            \pi(i) = \max(t+1,z).
        \end{equation}

        Note that 
        \begin{equation}
            \sigma_i(i)=Q(i,t')(i) = \max(i+1, t').
        \end{equation}

        When adding a new point to the standard plot of $\sigma_{i}$ to create $\pi$, this point must be added to the right of the $x = i$ line (otherwise the descent at $i$ of $\sigma_i$ becomes a descent at $i+1$ of $\pi$.) Thus, 
        \begin{equation}
            \pi(i)\in \{\sigma_i(i), \sigma_i(i)+1\},
        \end{equation}
        since the $(i, \sigma_i(i))$ point is now either the $\sigma_i(i)-$th or $(\sigma_i(i)+1)-$th point from the bottom.

        Thus, \begin{equation}\sigma_i(i) \ge \pi(i)-1= \max(t+1, z)-1 \ge t \ge i+1.\end{equation} 
        Since $\sigma_i(i) = \max(t', i+1)$ we have $\sigma_i(i) = t'\ge i+1$. If $t' = t$ we are done. Thus, we may assume $t'\ne t$ and so $\sigma_i(i) \ne t$ and thus $\sigma_i(i)\ge t+1$.
        We also have $\sigma_i(i)\le \pi(i) = \max(t+1, z)$, so if $z\le t+1$ we have $t' = \sigma_i(i) = t+1$ as desired. So we may assume $z > t+1$, which implies (see \eqref{eq:largez}) \begin{equation}\pi(i-1) = t.\label{eq:piim1}\end{equation}
        Since $t'\ge i+1,$ the permutations in $S_{k+1}$ containing $Q(i, t')$ with descent set $\{i\}$ are 
        \begin{equation}1, 2, \dots, \what{r},\dots, i, t'+1, r, \dots,\what{i}, \dots, \what{t'+1},\dots, k+1\end{equation} for $1\le r\le i$
        and $Q_{k+1}(i,t'),$ all of which violates \eqref{eq:piim1}.
        This is a contradiction.
    \end{proof}
\end{lm}

\begin{lm}\label{lm:symmshuffle}
    If $S_{k+1}(\Pi)$ is symmetric then $\Pi$ is a partial shuffle.
    \begin{proof}
        Our strategy is to use \Cref{lm:countconone} and \Cref{lm:countcontwo} to show that there are $t_1, \dots, t_{k-1}$ such that $\sigma_i = Q(i,t_i)$ for all $i$. We then use Lemmas \ref{lm:countcontwo}, \ref{lm:countconf}, and \ref{lm:countaddnodes}  
        to show that one can choose $t_1 = \dots = t_{k-1},$ which means $\Pi$ is a partial shuffle.

        The lemma can be computationally verified for $k= 4$. Thus, we assume $k\ge 5$.

        Let $Z:=S_{k+1}\sm S_{k+1}(\Pi)$ be the set of permutations in $S_{k+1}$ containing at least one pattern in $\Pi$. Since $S_{k+1}(\Pi)$ is symmetric, $Z$ must be symmetric.

        Let 
        \begin{equation}
            Z_t := \{\pi\in Z\mid \pi\text{ contains }\sigma_t\} = S_{k+1}\sm S_{k+1}\pa{\{\sigma_t\}}
        \end{equation}

        For $i < j$, recall the notation $A(Z)_{\{i, j+1\}}$ means permutations in $Z$ with descents at $i$ and $j+1$ (and possibly elsewhere.) By abuse of notation we will write this as $A(Z)_{i, j+1}.$ Note that for each permutation in $Z$, deleting some letter and then standardizing gives a permutation in $\Pi$ which has only one descent, so a permutation in $Z$ can have at most two descents. Thus, $A(Z)_{i, j+1}$ is the set of the permutations in $Z$ with descent set exactly at $\{i,j+1\}.$
        
        Note that for $\pi\in A(Z)_{i, j+1},$ $\del(\pi, p)$ will still have two descents unless we delete one of $i, i+1, j+1, j+2.$ And if we delete either $i, i+1$ and obtain a sequence with only one descent then the descent is at $j$ (due to reindexing). Similarly, if we delete either $j+1, j+2$ and obtain a sequence with only one descent then the descent is at $i$. Thus, we have if there is $\pi\in Z_t\cap A(Z)_{i, j+1}$ then $t\in \{i, j\}.$ That is, we have $A(Z)_{i, j+1}  = A(Z)_{i, j+1}\cap (Z_i\cup Z_j),$ and thus 
        \begin{equation}
            \abs*{A(Z)_{i, j+1}} = \abs*{A(Z)_{i, j+1}\cap Z_i} + \abs*{A(Z)_{i, j+1}\cap Z_j} - \abs*{A(Z)_{i, j+1}\cap Z_i\cap Z_j}
        \end{equation}
        by inclusion-exclusion principle.

        By \Cref{lm:countconone} we have $|A(Z)_{i,j+1}\cap Z_i| = |A(Z)_{i,j+1}\cap Z_j| = k,$ so we have  
        \begin{equation}\label{eq:agij}
            \abs*{A(Z)_{i, j+1}} = 2k - \abs*{A(Z)_{i, j+1}\cap Z_i\cap Z_j}
        \end{equation}
        and thus, $\abs*{A(Z)_{i, j+1}\cap Z_i\cap Z_j}$ does not depend on $i$ and $j$, since $\abs*{A(Z)_{i, j+1}}$ for $i < j$ does not depend on $i, j$ by harmonicity of $A(Z).$

        As in \Cref{lm:countcontwo}, define \begin{equation}
                P_i:= \{\std(\del(\sigma_i, i)), \std(\del(\sigma_i, i+1))\}.
            \end{equation}

        Since $\Des(\sigma_1) = \{1\}$ we have $\sigma_1(2) < \sigma_1(3) \dots < \sigma_1(k)$, and thus $\std(\del(\pi_1, 1)) = \iota_{k-1}$. Similarly, $\std(\del(\sigma_{k-1}, k)) = \iota_{k-1}$, so $P_1\cap P_{k-1}\ne\varnothing$. By \Cref{lm:countcontwo}, $A(Z)_{1, k}\cap Z_1\cap Z_{k-1}\ne\varnothing$, and so we must have  $A(Z)_{i, j+1}\cap Z_i\cap Z_j\ne\varnothing$ for all $i < j$, so $P_i\cap P_j\ne \varnothing$ for all $i < j$ by Lemma \ref{lm:countcontwo}. 

        Note that $\Des(\std(\del(\sigma_i, i))) = \Des(\del(\sigma_i, i))$ is either $\{i-1\}$ or $\varnothing$ since we have $\sigma_i(1) < \dots < \sigma_i(i), \sigma_i(i+1) < \dots < \sigma_i(k)$. Thus, $\std(\del(\sigma_i, i))$ is either $\iota_{k-1}$ or a permutation with descent set $\{i-1\}$. Similarly, $\std(\del(\sigma_i, i+1))$ is either $\iota_{k-1}$ or a permutation with descent set $\{i\}$.
        Therefore, for any $i$, a permutation in $P_i\sm \{\iota_{k-1}\}$, if exists, has descent set $\{i-1\}$ or $\{i\}$. This means that $P_i \cap P_j \subseteq \{\iota_{k-1}\}$ for $|i-j| > 1$.

        Thus, for $i > 2$ since $P_1\cap P_i\ne \varnothing$  we have $\iota_{k-1}\in P_i$; and for $i < k-2$ since $P_i\cap P_k\ne \varnothing$ we have $\iota_{k-1}\in P_i$. Since $k \ge 5$, for any $1 \le i\le k$ we have $i > 2$ or $i < k-2$ so for any $i$ we have $\iota_{k-1}\in P_i$. In particular, we have $\sigma_i$ contains $\iota_{k-1}$ for all $i$ and $P_i \cap P_j = \{\iota_{k-1}\}$ for $|i-j| > 1$.

    \begin{casesp}
        
    \item There exists $t$ such that $\sigma_1 = Q(1,t)$ and $\sigma_{k-1}  = Q(k-1,t)$. Note that either $t > 2$ or $t < k-1$, by replacing $\Pi$ with $\Pi^{cr}$ we may assume $t > 2.$ 
        Note that there is $t_{ij}$ such that $\sigma_i = Q(i, t_{ij})$ and  $\sigma_i = Q(j, t_{ij})$ for all $i, j$ satisfying $P_i\cap P_j = \{\iota_{k-1}\}$ since otherwise we have $\abs*{A(Z)_{i, j+1}\cap Z_i\cap Z_j}= 1 < 2 \le \abs*{A(Z)_{1, k}\cap Z_1\cap Z_{k-1}}$ by Lemma \ref{lm:countcontwo}. In particular, such $t_{ij}$ exists for $j - i > 1$. By \Cref{lm:p1p2} we have either there is $t_{12}$ such that $\sigma_1 = Q(1,t_{12})$ and $\sigma_2 = Q(2,t_{12})$, or $P_1\cap P_2 = \{\iota_{k-1}\}.$ But if $P_1\cap P_2 = \{\iota_{k-1}\}$ then there is also $t_{12}$ such that $\sigma_1 = Q(1,t_{12})$ and $\sigma_2 = Q(2,t_{12}).$ By \Cref{rmk:caseworkg}, we must have $t_{1j} = t$ for all $j > 1$, so $\sigma_j = Q(j,t)$ for all $j$. Thu, $\Pi$ is the partial shuffle 
        \begin{equation}
            (1, 2, \dots, \what{t}, \dots, k)\cshuffle (t).
        \end{equation}

        \item There is no $t$ such that $\sigma_1 = Q(1,t)$ and $\sigma_{k-1}  = Q(k-1,t).$ Since $(k-1)-1>1$, we have $P_1\cap P_{k-1} = \{\iota_{k-1}\}$. Thus, $|A(Z)_{1, k}\cap Z_1\cap Z_{k-1}| =1$ by \Cref{lm:countcontwo}. Thus, $|A(Z)_{i, j+1}\cap Z_i\cap Z_{j}| =1$ for all $i< j$. Thus, for any $i < j$, there is no $t$ such that $\sigma_i = Q(i,t)$ and $\sigma_j = Q(j,t)$ by Lemma \ref{lm:countcontwo}. By \Cref{rmk:caseworkg}, we can choose $t_i$ such that $\sigma_i = Q(i, t_i)$ for each $i$, and for each $i$, $t_i$ can be chosen in either $1$ or $2$ ways. But note that the possible choices of $t_i$ for different $i$ must be disjoint. Since there are $k-1$ indices $i$ and there are $k$ total choices, at most one of them has $2$ choices. 
        Thus, either $t_1$ has $1$ choice or  $t_{k-1}$ has $1$ choice. By replacing $\Pi$ with $\Pi^{cr}$ we may assume $t_1$ has $1$ choice, so $t_1 >2$.

        Recall the notation $A(Z)_{\{i\}}$ denotes the permutations in $Z$ with a descent at $i$ and possibly elsewhere, by abuse of notation we write this as $A(Z)_i$. Note that $A(Z)_{\{i\}, [k]\sm\{i\}}$ denotes the set of permutation in $Z$ whose descent set is exactly $\{i\}.$
        Recall that a permutation in $Z$ can have at most two descents, so we have 
        \begin{align}\label{eq:m1cal}
            \abs*{A(Z)_{1}} &= |A(Z)_{\{1\}, [k]\sm \{1\}}| + \sum_{i = 2}^{k}|A(Z)_{1,i}|\notag \\
            &= |A(Z)_{\{1\}, [k]\sm \{1\}}| + |A(Z)_{1,2}| + (k-2)(2k-1),
        \end{align}
        where we used \eqref{eq:agij} together with $|A(Z)_{i, j+1}\cap Z_i\cap Z_{j}| =1$ for $i < j$ to obtain that $\abs*{A(Z)_{i,j}} = 2k-1$ for $|i-j|>1.$
 
        Note that deleting a letter from $\pi\in A(Z)_{\{1\}, [k]\sm \{1\}}$ give a sequence that either have a descent at $1$ or no descents, so $A(Z)_{\{1\}, [k]\sm \{1\}}\sbs Z_1.$ Similarly, deleting a letter from $\pi\in A(Z)_{1,2}$ give a sequence that either have a descent at $1$ or at least two descents, so $A(Z)_{1,2}\sbs Z_1.$ We thus have $|A(Z)_{\{1\}, [k]\sm \{1\}}| = 2$ and $|A(Z)_{1,2}| = k-1$ by  \Cref{lm:countaddnodes} and \Cref{lm:countconf}.

        Thus, by \eqref{eq:m1cal} we have $\abs*{A(Z)_1} = 2 + (k-1) + (k-2)(2k-1) = 2k^2 - 4k +3$.

        On the other hand for $1 < i < k$ we have 
        \begin{align}
            |A(Z)_{i}| &= |A(Z)_{\{i\}, [k]\sm \{i\}}| + |A(Z)_{i-1,i}| + |A(Z)_{i, i+1}| + \sum_{|j-i|> 1}A(Z)_{i,j}\notag\\
            &= |A(Z)_{\{i\}, [k]\sm \{i\}}| + 2(k-1) + (k-3)(2k-1)\\
            &= |A(Z)_{\{i\}, [k]\sm \{i\}}| + 2k^2 - 5k +1\notag
        \end{align}
        where we used harmonicity of $A(Z)$ to obtain that $|A(Z)_{i,i+1}| = |A(Z)_{1,2}| = k-1.$

        Since $|A(Z)_{i}| = |A(Z)_{1}|$ we have $|A(Z)_{\{i\}, [k]\sm\{i\}}| = k+2$ for all $1 < i < k$. 
        Note that deleting an element from $\pi\in A(Z)_{\{i\}, [k]\sm \{i\}}$ gives a sequence with descent at $i$ or at $i-1$ or no descent, so we have $A(Z)_{\{i\}, [k]\sm\{i\}}\sbs Z_{i-1}\cup Z_i,$ so
        \begin{align}\label{eq:pii}
            k+ 2 &= |A(Z)_{\{i\}, [k]\sm \{i\}}|  = |A(Z)_{\{i\}, [k]\sm \{i\}}\cap Z_{i-1}| + \abs*{A(Z)_{\{i\}, [k]\sm \{i\}}\cap Z_{i}} \notag\\
            &- \abs*{A(Z)_{\{i\}, [k]\sm \{i\}}\cap Z_{i-1}\cap Z_i}
        \end{align}

        By Lemma \ref{lm:countaddnodes} we have $\abs*{A(Z)_{\{i\}, [k]\sm \{i\}}\cap Z_{i-1}} = k+2 - i$ and $\abs*{A(Z)_{\{i\}, [k]\sm \{i\}}\cap Z_{i}} = i+1$. Thus, by \eqref{eq:pii} we have $\abs*{A(Z)_{\{i\}, [k]\sm \{i\}}\cap Z_{i-1}\cap Z_i} = 1$ for $1 < i < k$.

        Note that $\sigma_1 = Q(1,t_1)$ for some $t_1 > 2$.
        If $\sigma_{i-1} = Q(i-1,t_1+i-2)$ for some $1 < i < k$ we will have that $\sigma_i = Q(i,t_1+i-1)$ by Lemma \ref{lm:inductstep}. 
        Thus, by induction we have $\sigma_{k-1} = Q(k-1,k+t_1-2)$, which is a contradiction since $k+t_1-2>k$.\qedhere
    \end{casesp}
    \end{proof}
\end{lm} 

\begin{proof}[Proof of Theorem \ref{thm:smallsaset}]
    Note that $(2)\implies (3)$ follows from \cite[Theorem 4.1]{bloomsagan}, \cite[Proposition 2.6]{HZP} and \Cref{cor:crpat}, and $(3)$ clearly implies $(1)$. Thus, it suffices to show $(1)\implies (2)$.

    Since $\Pi$ is symmetrically avoided, we have $S_k(\Pi) = S_k\sm \Pi$ is symmetric. But $S_k$ is symmetric, so $\Pi$ is symmetric. Thus, by Theorem \ref{thm:symmk} we have $\Pi$ satisfies one of the four conditions (a), (b), (c), (d). If it satisfies (a) then we are done. There are only finitely many sets of permutations satisfying (c) or (d), and one can computationally verify that none of these are symmetrically avoided. (All but $1$ of the sets $\Pi$ satisfying the first two cases (c) has $S_5(\Pi)$ or $S_6(\Pi)$ not symmetric, and the remaining set has $S_7(\Pi)$ not symmetric; and all except for $8$ of the sets $\Pi $ satisfying the first case of (d) has $S_7(\Pi)$ not symmetric, the remaining $8$ has $S_8(\Pi)$ not symmetric.)

    Thus, we only need to consider the case $\Pi$ satisfies (b). By taking complement, without loss of generality, assume $S =\{\sigma_1, \dots, \sigma_{k-1}\}$ with $\Des(\sigma_{i}) = \{i\}$. Then if $k\ge 4$ by Lemma \ref{lm:symmshuffle}, we have $\Pi$ is a partial shuffle. For $k = 3$ all set of the form $\{\sigma_1,\sigma_2\}$ with $\Des(\sigma_i) = \{i\}$ are partial shuffles or their complement, so the conclusion holds trivially. For $k\le 2$ we have $S_k = \{\iota_k, \delta_k\}$ so $\Pi\subseteq S_k = \{\iota_k, \delta_k\}$.
\end{proof}

\bibliographystyle{amsalpha}
\bibliography{Citations}

@inproceedings{S15,
  author       = {Bruce E. Sagan},
  title        = {Pattern avoidance and quasisymmetric functions},
  booktitle    = {13th International Conference on Permutation Patterns},
  year         = {2015},
  address      = {London, England},
  month        = jun # { } # 20,
  note         = {\url{https://users.math.msu.edu/users/bsagan/Slides/PaqPPh.pdf}},
}

@article {M25,
    AUTHOR = {Marmor, Avichai},
     TITLE = {Tight lower bound for pattern avoidance and symmetric
              functions},
   JOURNAL = {Israel J. Math.},
  FJOURNAL = {Israel Journal of Mathematics},
    VOLUME = {270},
      YEAR = {2025},
    NUMBER = {1},
     PAGES = {337--356},
      ISSN = {0021-2172,1565-8511},
   MRCLASS = {05},
  MRNUMBER = {4999802},
       DOI = {10.1007/s11856-025-2783-5},
       URL = {https://doi.org/10.1007/s11856-025-2783-5},
}

@article {ER17,
    AUTHOR = {Elizalde, Sergi and Roichman, Yuval},
     TITLE = {Schur-positive sets of permutations via products and grid
              classes},
   JOURNAL = {J. Algebraic Combin.},
  FJOURNAL = {Journal of Algebraic Combinatorics. An International Journal},
    VOLUME = {45},
      YEAR = {2017},
    NUMBER = {2},
     PAGES = {363--405},
      ISSN = {0925-9899,1572-9192},
   MRCLASS = {05E05 (05A05)},
  MRNUMBER = {3604061},
MRREVIEWER = {Michael\ Xinxin\ Zhong},
       DOI = {10.1007/s10801-016-0710-x},
       URL = {https://doi.org/10.1007/s10801-016-0710-x},
}

@book {stanley,
    AUTHOR = {Stanley, Richard P.},
     TITLE = {Enumerative combinatorics. {V}ol. 2},
    SERIES = {Cambridge Studies in Advanced Mathematics},
    VOLUME = {208},
   EDITION = {Second},
      NOTE = {With an appendix by Sergey Fomin},
 PUBLISHER = {Cambridge University Press, Cambridge},
      YEAR = {[2024] \copyright 2024},
     PAGES = {xvi+783},
      ISBN = {978-1-009-26249-1; 978-1-009-26248-4},
   MRCLASS = {05-02 (05A15 05E05 05E10 68R05)},
  MRNUMBER = {4621625},
MRREVIEWER = {Timothy\ Y.\ Chow},
}

@article {er,
    AUTHOR = {Elizalde, Sergi and Roichman, Yuval},
     TITLE = {Arc permutations},
   JOURNAL = {J. Algebraic Combin.},
  FJOURNAL = {Journal of Algebraic Combinatorics. An International Journal},
    VOLUME = {39},
      YEAR = {2014},
    NUMBER = {2},
     PAGES = {301--334},
      ISSN = {0925-9899,1572-9192},
   MRCLASS = {05A05 (05A15)},
  MRNUMBER = {3159254},
MRREVIEWER = {Yu\ Chen},
       DOI = {10.1007/s10801-013-0449-6},
       URL = {https://doi.org/10.1007/s10801-013-0449-6},
}

@article {ggc,
    AUTHOR = {Albert, Michael H. and Atkinson, M. D. and Bouvel, Mathilde
              and Ru\v skuc, Nik and Vatter, Vincent},
     TITLE = {Geometric grid classes of permutations},
   JOURNAL = {Trans. Amer. Math. Soc.},
  FJOURNAL = {Transactions of the American Mathematical Society},
    VOLUME = {365},
      YEAR = {2013},
    NUMBER = {11},
     PAGES = {5859--5881},
      ISSN = {0002-9947,1088-6850},
   MRCLASS = {05A05 (05A15)},
  MRNUMBER = {3091268},
MRREVIEWER = {Antonio\ Bernini},
       DOI = {10.1090/S0002-9947-2013-05804-7},
       URL = {https://doi.org/10.1090/S0002-9947-2013-05804-7},
}

@article {bloomsagan,
    AUTHOR = {Bloom, Jonathan S. and Sagan, Bruce E.},
     TITLE = {Revisiting pattern avoidance and quasisymmetric functions},
   JOURNAL = {Ann. Comb.},
  FJOURNAL = {Annals of Combinatorics},
    VOLUME = {24},
      YEAR = {2020},
    NUMBER = {2},
     PAGES = {337--361},
      ISSN = {0218-0006,0219-3094},
   MRCLASS = {05E05 (05A05)},
  MRNUMBER = {4110402},
MRREVIEWER = {Yan\ Zhuang},
       DOI = {10.1007/s00026-020-00492-6},
       URL = {https://doi.org/10.1007/s00026-020-00492-6},
}

@article {gesselreutenauer,
    AUTHOR = {Gessel, Ira M. and Reutenauer, Christophe},
     TITLE = {Counting permutations with given cycle structure and descent
              set},
   JOURNAL = {J. Combin. Theory Ser. A},
  FJOURNAL = {Journal of Combinatorial Theory. Series A},
    VOLUME = {64},
      YEAR = {1993},
    NUMBER = {2},
     PAGES = {189--215},
      ISSN = {0097-3165,1096-0899},
   MRCLASS = {05A15 (05E05)},
  MRNUMBER = {1245159},
       DOI = {10.1016/0097-3165(93)90095-P},
       URL = {https://doi.org/10.1016/0097-3165(93)90095-P},
}

@article {HZP,
    AUTHOR = {Hamaker, Zachary and Pawlowski, Brendan and Sagan, Bruce E.},
     TITLE = {Pattern avoidance and quasisymmetric functions},
   JOURNAL = {Algebr. Comb.},
  FJOURNAL = {Algebraic Combinatorics},
    VOLUME = {3},
      YEAR = {2020},
    NUMBER = {2},
     PAGES = {365--388},
      ISSN = {2589-5486},
   MRCLASS = {05E05 (05E16)},
  MRNUMBER = {4099000},
MRREVIEWER = {Markus\ E.\ Fulmek},
       DOI = {10.5802/alco.96},
       URL = {https://doi.org/10.5802/alco.96},
}

@incollection {gessel,
    AUTHOR = {Gessel, Ira M.},
     TITLE = {Multipartite {$P$}-partitions and inner products of skew
              {S}chur functions},
 BOOKTITLE = {Combinatorics and algebra ({B}oulder, {C}olo., 1983)},
    SERIES = {Contemp. Math.},
    VOLUME = {34},
     PAGES = {289--317},
 PUBLISHER = {Amer. Math. Soc., Providence, RI},
      YEAR = {1984},
      ISBN = {0-8218-5029-6},
   MRCLASS = {05A17 (20C30)},
  MRNUMBER = {777705},
MRREVIEWER = {J.\ D\'esarm\'enien},
       DOI = {10.1090/conm/034/777705},
       URL = {https://doi.org/10.1090/conm/034/777705},
}

\end{document}